\documentclass[12pt]{amsproc}
\usepackage{amsthm,amssymb,amsmath,amsopn,amsfonts,etoolbox,mathptmx}
\usepackage{stmaryrd,calc,mathrsfs,enumerate}
\usepackage{color}
\usepackage{marginnote}

\usepackage{geometry}
\usepackage{prerex}

\makeatletter
\patchcmd{\@settitle}{\uppercasenonmath\@title}{}{}{}
\patchcmd{\@setauthors}{\MakeUppercase}{\scshape}{}{}
\patchcmd{\section}{\scshape}{\bfseries}{}{}
\renewcommand{\@secnumfont}{\bfseries}
\patchcmd{\abstract}{\scshape\abstractname}{\textbf{\abstractname}}{}{}
\makeatother

\usepackage{hyperref} 
\hypersetup{
	colorlinks=true,
	citecolor=blue,
	linkcolor=red
}

\newtheorem{theorem}{Theorem}[section]
\newtheorem{proposition}{Proposition}[section]
\newtheorem{lemma}{Lemma}[section]
\newtheorem{cor}{Corollary}[section]

\theoremstyle{definition}
\newtheorem{definition}{Definition}[section]
\newtheorem{ex}{Example}[section]

\theoremstyle{remark}
\newtheorem{remark}{Remark}[section]

\newcommand{\RRe}{\mathrm{Re}}
\newcommand{\Log}{\mathrm{Log}}
\newcommand{\Res}{\mathrm{Res}\,}

\newcommand{\sign}{\mathrm{sign}}
\newcommand{\supp}{\mathrm{supp}}
\newcommand{\im}{\mathrm{im}}

\begin{document}
\title[Geometry of generalized amoebas]{Geometry of generalized amoebas}
\author{Yury~V. Eliyashev}
\address{Institute of Mathematics and Computer Science, Siberian Federal University, 79 
Svobodny pr., 660041 Krasnoyarsk, Russia}\address{Moscow Institute of Electronics and Mathematics, National Research University Higher School of Economics, 34 Tallinskaya, 123592 Moscow, Russia} \email{eliyashev@gmail.com}
\thanks{The research for this work was carried out in Siberian Federal University and was supported by grant of the Ministry of Education and Science of the Russian Federation № 1.2604.2017/PCh.
 The author is a Young Russian Mathematics award winner and would like to thank
its sponsors and jury.}
\begin{abstract}
Recently Krichever proposed a generalization of the amoeba and the
Ronkin function of a plane algebraic curve. In our paper higher-dimensional version of this generalization is studied.
We translate to the generalized case different geometric results known in the standard amoebas case. 
\end{abstract}
\maketitle

\section*{Introduction.}

The \emph{amoeba} $\mathcal{A}_X$ of a closed complex analytic set $X$ in $(\mathbb{C}^*)^m$ is, by definition, the image in $\mathbb{R}^m$
of this set under the map $$\Log:(z_1,\dots,z_m)\rightarrow (\log|z_1|,\dots,\log|z_m|).$$
The amoeba $\mathcal{A}_F$ of a Laurent polynomial $F(z)$ in $(\mathbb{C}^*)^m$ is
the image of the hypersurface $\{F(z)=0\}$ 
under the map $\Log.$ In this paper we will usually call these objects classical amoebas. 
The notion of amoeba was introduced by Gelfand, Kapranov and Zelevinsky in \cite{GKZ}. 
This object was widely studied in the last decades and found different applications in complex geometry and tropical geometry.

Recently Krichever introduced the notion of generalized amoeba of a complex curve \cite{Kr}. To define this object one need to 
choose on a given smooth complex curve $C$ a pair of meromorphic differentials $\omega_1,\omega_2$ with some additional properties. 

In our paper we define the notion of a generalized amoeba in higher dimensions.
The main purpose of this paper is to extend results known in the classical case to the multidimensional generalized amoeba case.
Some of our results are proved similarly to the classical ones. 
 On the other hand, some proofs of theorems in the classical amoeba case are based on direct manipulations with the Laurent polynomial $F(z)$ of an amoeba.
Since in the generalized case we have no analog of this polynomial, we have to use different techniques.

Let us recall some facts about the classical amoebas of a Laurent polynomial $F(z).$
We denote by $\Upsilon$ the set of connected components of $\mathbb{R}^m\setminus \mathcal{A}_F.$
Any component $C\in \Upsilon$ is a convex set in $\mathbb{R}^m$. 
Let us define the \emph{Ronkin function} $R_F:\mathbb{R}^m \rightarrow \mathbb{R}$ of the amoeba $\mathcal{A}_F$ as follows
$$R_F(x)=\frac{1}{(2 \pi i)^m} \int_{\Log^{-1}(x)} \log|F(z)| \frac{d z_1}{z_1}\wedge \dots\wedge \frac{d z_m}{z_m}.$$
The Ronkin function $R_F$ is convex and its restriction to any connected components of $\mathbb{R}^m\setminus \mathcal{A}_F$ is an affine function.
Therefore one can define the map $\nu_F$ from the set $\Upsilon$ to $\mathbb{R}^m$ as $\nu_F(C)=\nabla R_f(x),$ where $C\in \Upsilon$ and $x\in C,$ this map is called the \emph{order map.}
In fact, $\nu_F$ is injective, integer valued and its image is contained in the Newton polytope $N_F$ of the polynomial $F.$
Moreover, the convex hull of the image of $\nu_F$ coincides with  $N_F$.
One of the main facts about geometry of $\mathcal{A}_F$ is that the recession cone of a component $C\in \Upsilon$ is equal 
to the normal cone to the Newton polytope $N_F$ at the point $\nu_F (C)$ \cite{FPT}.

Let us introduce the main objects of our study.
Suppose that $X$ is an $n-$dimensional compact complex manifold and $V$ is a simple normal crossing divisor on $X.$
Denote by $H^0(X,\Omega^1(\log V))$ the space  of holomorphic $1$-forms on $X\setminus V$ with
logarithmic poles along $V.$ A form $\psi\in H^0(X,\Omega^1(\log V))$ is called  an \emph{imaginary normalized holomorphic differential} if
$d \psi=0$ and for any cycle $\gamma\in H_1(X\setminus V, \mathbb{R})$ holds
$$\int_\gamma \psi \in i \mathbb{R}.$$
Let $\omega=(\omega_1,\dots,\omega_m)$ be a vector of imaginary normalized holomorphic differentials, then 
we define the following map from $X\setminus V$ to $\mathbb{R}^m$ 
$$\Log_{\omega,p_0}(p)=(\RRe \int^p_{p_0} \omega_1,\dots, \RRe \int^p_{p_0} \omega_m),$$
where $p,p_0\in X\setminus V.$ The choice of the point $p_0$ does not play much role, so we will usually write just $\Log_{\omega}$ instead of 
$\Log_{\omega,p_0}.$

The \emph{generalized amoeba} $\mathcal{A}_\omega$ associated with $\omega=(\omega_1,\dots,\omega_m)$ and the pair $X,V$ is the image of  the map $\Log_{\omega}$ in $\mathbb{R}^m.$
We will show that generalized amoebas have properties similar to
classical amoebas.

This paper is organized as follows. In the first section we give the definition of a generalized amoeba.
We describe the set of critical values of the map $\Log_{\omega,p_0}$ and the asymptotic behavior of an generalized amoeba.
Also we give a complete description of the space of imaginary normalized holomorphic differentials 
when  $X$ is a K\"{a}hler manifold with a fixed simple normal crossing divisor $V$ on it.

In the second section we study generalized amoebas when $m=n+1$ (here $n=\dim X$ and $m=\dim \mathbb{R}^m$).
This case is a generalized analog of an amoeba of a hypersurface in $(\mathbb{C}^*)^m.$ 
Geometry of classical amoebas of hypersurfaces in  $(\mathbb{C}^*)^m$ is the most studied and well-described topic in this area.

Consider the differential forms
$$\Omega_j=\omega_1\wedge\dots\wedge\omega_{j-1}\wedge\omega_{j+1}\wedge\dots\wedge\omega_{m}, \; j\in\{1,\dots,m\}.$$
We say that the generalized amoeba $\mathcal{A}_\omega$ satisfies the \emph{nondegeneracy condition}
if there exits $j$ such that $\Omega_j\not\equiv 0$ on $X.$ 
When we deal with the $m=n+1$ case we assume that the generalized amoeba $\mathcal{A}_\omega$ satisfies the nondegeneracy condition.

Our main tool to study the $m=n+1$ case is the Ronkin function of a generalized amoeba.
Because we don't have the polynomial $F$ we can't use the classical definition.
We denote the \emph{Ronkin function} of the generalized amoeba $\mathcal{A}_\omega$ by $R_\omega.$
Let $\phi(x)$ be a smooth function with a compact support and let $d x = d x_1 \wedge \dots \wedge d x_m.$ 
 Then we say that a function $R_\omega$ is the Ronkin function of the generalized amoeba $\mathcal{A}_\omega$
if 
$$\frac{\partial^2}{\partial x_k \partial x_j}R_\omega [\phi(x) d x]=
\frac {(-1)^{j+k+\frac{m(m-1)}{2}}}{2^n} \RRe \int_{X\setminus V} \frac{1}{(2 \pi i)^n} 
\Log^*_\omega (\phi) \Omega_{j}\wedge\overline{\Omega}_{k},$$
for any smooth function $\phi(x)$ with a compact support. 
Here $\frac{\partial^2}{\partial x_k \partial x_j}R_\omega [\phi(x) d x]$ denotes the value of the current 
$\frac{\partial^2}{\partial x_k \partial x_j}R_\omega$ acting on the test form $\phi(x) d x.$
In other words, it equals
$$\frac{\partial^2}{\partial x_k \partial x_j}R_\omega [\phi(x) d x] = \int_{\mathbb{R}^m} 
(\frac{\partial^2}{\partial x_k \partial x_j}R_\omega) \phi(x) d x : = \int_{\mathbb{R}^m} R_\omega
(\frac{\partial^2}{\partial x_k \partial x_j}\phi(x)) d x.$$
We prove that this function exists, it is convex, and the restriction of $R_\omega$ to any connected component of
$\mathbb{R}^m\setminus \mathcal{A}_\omega$ is an affine function. Because the Ronkin function function 
is not necessarily smooth we calculate its derivatives in the sense of generalized functions.

The notions of tropical supercurrents and superform was introduced in the tropical geometry setting \cite{Lag}.
Tropical superforms is a tropical analog of differential forms on complex manifolds, in particular they were used to construct a tropical analog
of de Rham cohomology theory \cite{JSS}.
The language of superforms is useful in our study. In particular, we interpret the Hessian of the Ronkin function as a positive $(1,1)-$supercurrent
$$\sum^m_{j,k=1} \frac{\partial^2}{\partial x_k \partial x_j}R_\omega d x_j \otimes d x_k.$$
Then the existence and properties of the Ronkin function is implied 
by the tropical analog of the $\partial\overline{\partial}-$lemma.

Now, let us denote by $\Upsilon$ the set of connected component of $\mathbb{R}^m\setminus \mathcal{A}_\omega$.
We prove that each connected component $C\in\Upsilon$ is convex.
In the generalized case we define the order map $\nu_\omega: \Upsilon \rightarrow \mathbb{R}^m$
in the same way as it was defined in the classical case, i.e.,
$$\nu_\omega(C)= \nabla R_\omega(x),$$ where $x$ is any point from $C.$
 Since we don't have the polynomial $F$ we can't use the classical definition of the Newton polytope, we define the \emph{Newton polytope} 
 $N_\omega$ of the amoeba $\mathcal{A}_\omega$ to be equal to the convex hull of the image 
of the order map $\nu_\omega$.
We prove that $\nu_\omega$ is injective and the recession cone of a component $C\in \Upsilon$ is equal 
to the normal cone to the Newton polytope $N_\omega$ at the point $\nu_\omega (C).$  These are generalizations of the classical case 
statements.
Also we prove several statements about the Monge-Amp\`{e}re measure associated with $R_\omega,$ these statements are generalization 
of results from \cite{PR}.

In the last section we give a coordinate-free construction of the generalized amoeba and
reformulate results of the first two section in a  coordinate-free form, it gives more abstract view on this subject.

Also we would like to mention another paper on generalized amoebas.
The tropical limit of generalized amoebas of complex curves was studied in \cite{Lan}, this paper describes
interesting relations
between generalized amoebas, tropical curves and geometry of the moduli space of complex curves. 

\section{Definition of a generalized amoeba and its general properties.}
\subsection{Definition of  a generalized amoeba}
\begin{definition} Let $X$ be a closed analytical subset of the algebraic torus $(\mathbb{C^*})^m.$
Consider the map $$\Log:(\mathbb{C^*})^m \rightarrow \mathbb{R}^m,$$
$$\Log(z_1,\dots, z_m)=(\log|z_1|,\dots, \log|z_m|).$$
The \emph{amoeba} $\mathcal{A}_X \subset \mathbb{R}^m$ of $X$ is the image of $X$ under the map $\Log$, i.e., $$\mathcal{A}_X = \Log(X).$$
usually we will call it a classical amoeba. 
\end{definition}

Let $X$ be a smooth compact complex $n-$dimensional manifold.
Consider $V=\bigcup_j D_j$ a simple normal crossing divisor on $X$, i.e.,
all irreducible components $D_j$ are smooth and for any point $x\in V$ there are local 
coordinates $z_1,\dots,z_n$ on $X$ such that $V$ is locally defined by the equation 
$z_1\cdot {\dots} \cdot z_k=0$ in some neighborhood $U\subset X$ of $x.$
\begin{definition} Let $\Omega^1_{X}(\log V)$ be a
locally free $\mathcal{O}_{X}$-module, that is generated on an open neighborhood $U$
by the differentials $$\frac{d z_1}{z_1},\dots,\frac{d z_k}{z_k},d z_{k+1},\dots, d z_n,$$
where  $z_1\cdot {\dots} \cdot z_k=0$ is a local equation of $V$ in $U.$ In other words, any element
$\omega\in\Omega^1_{X}(\log V)(U)$ has the following form
$$\omega=\sum^k_{j=1} f_j(z) \frac{d z_j}{z_j} + \sum^n_{j=k+1} f_j(z) d z_j,$$
where $f_j(z)\in\mathcal{O}_{X}(U).$
This sheaf is called the \emph{sheaf of logarithmic $1$-forms along $V$ over $X.$} 
We denote by $\Omega^k_{X}(\log V)$ the $k$-th exterior power of the sheaf $\Omega^1_{X}(\log V),$
i.e., $$\Omega^k_{X}(\log V)=\Omega^1_{X}(\log V)\wedge\dots \wedge\Omega^1_{X}(\log V).$$ This sheaf is called the \emph{sheaf of logarithmic $k$-forms along $V$ over $X.$} 
\end{definition}

\begin{definition} Consider a $1$-form $\omega\in H^0(X,\Omega^1_{X}(\log V))$ such that $d \omega=0$ on $X\setminus V.$
The $1$-form $\omega$ is called \emph{imaginary normalized holomorphic differential} if
for any cycle $\gamma\in H_1(X\setminus V, \mathbb{R})$ holds
$$\int_\gamma \omega \in i \mathbb{R}.$$
\end{definition}

Given a vector of imaginary normalized holomorphic differentials $\omega=(\omega_1,\dots,\omega_m)$ and a point $p_0\in X\setminus V,$ we define the map
$$\Log_{\omega,p_0}: X\setminus V \rightarrow \mathbb{R}^m$$ as follows
$$\Log_{\omega,p_0}(p)=(\RRe \int^p_{p_0} \omega_1,\dots, \RRe \int^p_{p_0} \omega_m).$$

Because the form $\omega_j$ is an imaginary normalized holomorphic differential, the real part of $\int^p_{p_0} \omega_j$ does not depend on the 
choice of a path of integration, hence the map $\Log_{\omega,p_0}$ is well-defined. Obviously, for  $p_0,p'_0\in X\setminus V$ the 
corresponding maps $\Log_{\omega,p_0},$ $\Log_{\omega,p'_0}$ differ from each other by the shift by the vector
$$(\RRe \int^{p'_0}_{p_0} \omega_1,\dots, \RRe \int^{p'_0}_{p_0} \omega_m).$$
Therefore the choice of the point $p_0$ does not play much role, we consider the maps $\Log_{\omega,p_0}$ for the different choices 
of $p_0$ to be equivalent, and we will usually write just $\Log_{\omega}$ instead of $\Log_{\omega,p_0}.$

Denote by $\mathrm{Sing} \; \omega_j$ the union of all $D_k$ such that $\omega_j$ has a pole along $D_k.$
\begin{definition}Let $\omega=(\omega_1,\dots,\omega_m)$ be a vector of imaginary normalized holomorphic 
differentials such that $\bigcup_j\mathrm{Sing} \omega_j = V.$
The \emph{generalized amoeba} $\mathcal{A}_\omega\subset \mathbb{R}^m$ associated with $\omega$ and the pair $X,V$
is the image of the map $\Log_\omega:X\setminus V \rightarrow \mathbb{R}^m.$
\end{definition}

\begin{ex}
Let $X$ be a smooth algebraic variety in $(\mathbb{C}^*)^m,$ we may think that there is a smooth compactification $\overline{X}$ of $X$ such that 
$\overline{X}\setminus X$ is a simple normal crossing divisor.
Choose $\omega=(\omega_1,\dots,\omega_m)$ as follows $$\omega_j=\frac{d z_j}{z_j }|_X.$$
Then $\Log_\omega$ is equal to $\Log|_{X}$ up to the shift by a constant vector, therefore the generalized amoeba 
$\mathcal{A}_\omega$ coincides with the classical amoeba $\mathcal{A}_X$ after the translation by this vector.
\end{ex}

The generalized amoebas of one-dimensional complex manifolds was introduced in \cite{Kr}. Let us consider the main object of study of \cite{Kr}.
\begin{ex}
Let $C$ be a smooth algebraic curve of genus $g$ with $s$ distinct 
marked points $p_1,\dots,p_s.$ For any set of numbers $a_1,\dots,a_s\in \mathbb{R}$ such 
that $\sum^s_{j=1} a_j=0$ there is a unique imaginary normalized holomorphic differential 
$\omega$ such that it has poles of order $\leq 1$ at the points $p_j$ and
$\mathrm{res}_{p_j} \omega = a_j.$ Let $\omega_1$ and $\omega_2$ be such differentials then there is the logarithmic map
$$\Log_{(\omega_1,\omega_2)}: C\setminus (p_1\cup\dots\cup p_s) \rightarrow \mathbb{R}^2.$$ The image of this map
is called the \emph{generalized amoeba of a complex curve with marked points}.
\end{ex}

\begin{definition} Let $X$ be a closed analytical subset of the algebraic torus $(\mathbb{C^*})^m.$
Consider the map $$\mathrm{Arg}:(\mathbb{C^*})^m \rightarrow (S^1)^m=\mathbb{R}^m/ 2\pi\mathbb{Z}^m,$$
$$\mathrm{Arg}(z_1,\dots, z_m)=(\mathrm{Arg} (z_1),\dots, \mathrm{Arg} (z_m)),$$
where $\mathrm{Arg} (z) \in \mathbb{R}/ 2 \pi\mathbb{ Z}$ is the argument of a complex number $z.$
The \emph{coamoeba} $\mathcal{C}_X$ of $X$ is the image of $X$ under the map $\mathrm{Arg}$, i.e., $\mathcal{C}_X=\mathrm{Arg}(X).$
\end{definition}
\begin{remark}In some sense coamoebas play a complementary role to amoebas. One can define a generalized 
coamoeba and some kind of hybrids between a coamoeba and an amoeba in a very similar fashion.
Let $\Lambda_r$ be a lattice of rank $r$ in $\mathbb{R}^m.$
Let $\omega=(\omega_1,\dots,\omega_m)$ be a vector of closed holomorphic $1-$form on $X\setminus V$ 
such that for any cycle $\gamma\in H_1(X\setminus V, \mathbb{Z})$ holds
$$\int_\gamma \omega \in \Lambda_r + i \mathbb{R}^m \subset \mathbb{C}^m=\mathbb{R}^m + i \mathbb{R}^m .$$
Then for a fixed point $p_0\in X\setminus V$ one can defined
the map $$\Log_\omega: X\setminus V \rightarrow \mathbb{R}^m/ \Lambda_r \simeq (S^1)^{r}\times \mathbb{R}^{m-r},$$
$$\Log_\omega(p)=(\RRe \int^p_{p_0} \omega_1,\dots, \RRe \int^p_{p_0} \omega_m) \; \mathrm{mod} \; \Lambda_r.$$
For a smooth algebraic variety $X$ in $(\mathbb{C}^*)^m,$ the vector $\omega=(\omega_1,\dots,\omega_m),$
$\omega_j=\frac{1}{i} \frac{d z_j}{ z_j}|_{X},$ and the lattice $\Lambda_m= 2 \pi \mathbb{Z}^m$
we obtain a classical coamoeba.
In this article we don't study generalized coamoebas, but it seems that one can translate some properties of 
classical coamoebas and of generalized amoebas to the  generalized coamoebas case.
\end{remark}

 \begin{proposition} \label{pr.closed}
 The map $\Log_\omega$ is closed and proper. In particular, the generalized amoeba $\mathcal{A}_\omega$ is a closed set in $\mathbb{R}^m$.
\end{proposition}
\begin{proof}
Let us consider the one-point compactification of $\mathbb{R}^m,$ that is $S^m=\mathbb{R}^m \cup \infty.$ 
Observe that $||\Log_\omega(p)|| \rightarrow + \infty$ as $p \rightarrow p'\in V.$
Therefore there is a continuous map $\widetilde{\Log}_\omega: X \rightarrow S^m$ such that 
$\widetilde{\Log}_\omega|_{X\setminus V} = \Log_\omega$ and $\widetilde{\Log}_\omega(p)=\infty$ for $p\in V.$ 
The closed map lemma states that every continuous map $f : X \rightarrow Y$ from a compact space $X$ to a Hausdorff space $Y$ is closed and proper.
Therefore $\widetilde{\Log}_\omega$ is closed and proper.

Let $C$ be a closed set in $X\setminus V,$ consider $C$ as a subset of $X.$ Then the image of the closure $\overline{C}$ of $C$ in $X$ under the map 
$\widetilde{\Log}_\omega$ is a closed set in $S^m,$ whence $\widetilde{\Log}_\omega (\overline{C}) \cap \mathbb{R}^m$ is a closed set in $\mathbb{R}^m.$ 
Since $\widetilde{\Log}_\omega (\overline{C}) \cap \mathbb{R}^m = \Log_\omega(C),$ we see that
$\Log_\omega(C)$ is a closed set in $\mathbb{R}^m$ and $\Log_\omega$  is a closed map.

Let $C$ be a compact set in $\mathbb{R}^m,$ then $\widetilde{\Log}_\omega^{-1} (C)$ is a compact set in $X.$
Moreover $\widetilde{\Log}_\omega^{-1} (C)=\Log_\omega^{-1} (C),$ thus $\Log_\omega^{-1} (C)$ is a compact set in $X\setminus V$ and 
$\Log_\omega$ is a popper map. 

In particular, the amoeba $\mathcal{A}_\omega = \Log_\omega(X\setminus V)$ is a closed set in $\mathbb{R}^m$.
\end{proof}

\subsection{Critical points of the logarithmic map}
Now we are going to describe critical points of the map $\Log_\omega$.
First let us introduce some notation. For any point $z\in X\setminus V$ there is a holomorphic map
$$\Log^{\mathbb{C}}_{\omega,z}(p)=(\int^p_{z} \omega_1,\dots, \int^p_{z} \omega_m) \in \mathbb{C}^m,$$
this map is well-defined in a small neighborhood of the point $z.$
Consider its differential $$d_z \Log^{\mathbb{C}}_{\omega,z} : T_z (X\setminus V) \rightarrow  T_{0}\mathbb{C}^m\simeq \mathbb{C}^m$$ at the point $z.$
Let $L_z$ denote the image of $d_z \Log^{\mathbb{C}}_{\omega,z},$ we consider $L_z$ as a complex linear subspace of $\mathbb{C}^m.$

 \begin{proposition} \label{lcrit}
The rank of the differential of $\Log_{\omega}$ at a point $z\in X \setminus V$ equals
$$2 \dim_\mathbb{C} L_z - \dim_\mathbb{C} L_z \cap \overline{L_z}, $$
where $\overline{L_z}$ is the complex linear space conjugated to $L_z.$
The point $z\in X \setminus V$ is a critical point of the map $\Log_\omega$ if and only if
\begin{itemize}
  \item $L_z \cap \overline{L_z} \neq 0$  or  $\dim_\mathbb{C} L_z<n$, when $2n\leq m;$ 
  \item $2 \dim_\mathbb{C} L_z - \dim_\mathbb{C} L_z \cap \overline{L_z}< m,$ when $2n\geq m.$
\end{itemize}
\end{proposition}
This proposition is an extension to the generalized case of Theorem 2 from \cite{BT} and of Theorem 3.2 from \cite{MN}.

\begin{proof}
Given a point $z\in X\setminus V,$ then for any point $p$ in a small neighborhood of $z$ we have
$$\Log_{\omega,p_0}(p)=\RRe (\Log^{\mathbb{C}}_{\omega,z}(p)) + \Log_{\omega,p_0}(z).$$
Whence the differential $$d_z \Log_{\omega,p_0} : T_z (X\setminus V) \rightarrow  T_{x}\mathbb{R}^m \simeq \mathbb{R}^m,$$ where $x=\Log_{\omega,p_0}(z),$
can be represented as the composition of two maps
$$d_z \Log_{\omega,p_0}= \RRe \circ d_z \Log^{\mathbb{C}}_{\omega,z} 
:T_z (X\setminus V) \stackrel{d_z \Log^{\mathbb{C}}_{\omega,z}}{\rightarrow} \mathbb{C}^m \stackrel{\RRe}{\rightarrow} \mathbb{R}^m.$$
The rank of $d_z \Log_{\omega,p_0}$ is the dimension of $\im d_z \Log_{\omega,p_0},$ it is equal to
$$\dim_{\mathbb{R}} \im d_z \Log_{\omega,p_0}= \dim_{\mathbb{R}} L_z - \dim_{\mathbb{R}} (L_z\cap \ker \RRe).$$
Obviously, $v\in \ker \RRe = i \mathbb{R}^m \subset \mathbb{C}^m$ if and only if $\overline{v}=-v.$ 
So, if $v\in L_z\cap \ker \RRe,$ then $v\in \overline{L_z}$ and $v\in L_z\cap\overline{L_z}\cap \ker \RRe.$ 
Let us show that $$\dim_{\mathbb{R}} L_z\cap\overline{L_z}\cap \ker \RRe=\dim_{\mathbb{C}} L_z \cap \overline{L_z}.$$ 
Consider two \emph{real} linear subspaces of $L_z \cap \overline{L_z}:$
$$V_1=\{v\in L_z \cap \overline{L_z}: \overline{v}=v\},$$
$$V_2=\{v\in L_z \cap \overline{L_z}: \overline{v}=-v\}.$$
Obviously, $V_2=i V_1,$ i.e., $v\in V_1$ iff $i v\in V_2.$
Take a vector $v\in L_z \cap \overline{L_z}.$ Then $\overline{v}\in L_z \cap \overline{L_z}$
and there is a decomposition $$v= \frac{v+\overline{v}}{2}+\frac{v-\overline{v}}{2}, \frac{v+\overline{v}}{2}\in V_1,\frac{v-\overline{v}}{2}\in V_2,$$
therefore $L_z \cap \overline{L_z}=V_1\oplus V_2=V_1\oplus i V_1.$ Observe that $V_2=L_z\cap\overline{L_z}\cap\ker \RRe,$ 
whence
$$\dim_\mathbb{R}L_z\cap\overline{L_z}\cap\ker \RRe=\dim_\mathbb{R} V_2 = \frac{1}{2} \dim_\mathbb{R} L_z 
\cap \overline{L_z}= \dim_\mathbb{C} L_z \cap \overline{L_z}.$$
Hence the rank of the map $\Log_{\omega,p_0}$ at the point $z$ is equal to $$\dim_{\mathbb{R}} 
\im d_z \Log_{\omega,p_0}= 2 \dim_\mathbb{C} L_z - \dim_\mathbb{C} L_z \cap \overline{L_z}.$$ 

The point $z$ is critical iff the rank of the differential of the map $\Log_{\omega,p_0}$ 
is lesser then $$\min(\dim_\mathbb{R} \mathbb{R}^m, \dim_\mathbb{R} X\setminus V).$$
Using the formula for the rank we obtain the following condition
$$2 \dim_\mathbb{C} L_z - \dim_\mathbb{C} L_z \cap \overline{L_z}< \min(m, 2n),$$ 
this gives us the statement of the proposition.
\end{proof}

\begin{remark} Let us write down $L_z$ explicitly. Suppose $z_1,\dots,z_n$ are local coordinates in a neighborhood of
a point in $X\setminus V.$ Then locally we have $$\omega_j=\sum^n_{k=1}f_{jk}(z) d z_k.$$
The space $L_{z}$ is a $\mathbb{C}-$linear span of the vectors $$\phi_k(z)=(f_{1k}(z),\dots,f_{mk}(z)),k=1,\dots,n.$$
\end{remark}

\begin{remark} Let $\widetilde{X}\subset X\setminus V$ be a set of points $z\in X$ such that 
$\dim_{\mathbb{C}} L_z=n.$ If $\widetilde{X}\neq \emptyset,$  then
it is a dense open subset. It is easy to check that
$\widetilde{X}\neq \emptyset$ iff $n\leq m$ and the amoeba satisfies nondegeneracy 
condition (see below,  Definition \ref{df.ndc}).   Let us define the \emph{logarithmic Gauss map} as follows
$$G: \widetilde{X}\rightarrow \mathrm{Gr}(n,m),$$
 $$G(z)=[L_z]\in \mathrm{Gr}(n,m),$$
 where $[L_z]$ is the class of the vector subspace $L_z\subset \mathbb{C}^m$ in the Grassmannian $\mathrm{Gr}(n,m).$
 Critical points of the logarithmic map in the classical case
 were studied in terms of the logarithmic Gauss map \cite{BT},\cite{MN}.
\end{remark}

\subsection{Imaginary normalized holomorphic differentials on K\"{a}hler manifolds}
Let $\omega$ be a $1$-form $\omega\in H^0(X,\Omega^1_{X}(\log V)),$ $V=\bigcup_{j} D_j,$ such that $d \omega=0.$
\emph{ The residue $\Res_{D_j} \omega$ of the from $\omega$ along the divisor $D_j$}, by definition, is equal to
$$\Res_{D_j} \omega = \frac{1}{2 \pi i}\int_{\gamma_j} \omega,$$ where $\gamma_j$ is the canonically oriented
boundary of a small complex disc transversal to the hypersurface $D_j,$ i.e., let $p$ be an arbitrary point from 
$ D_j\setminus \bigcup_{q\neq j} D_q,$ and $z_1,\dots,z_n$ be a local coordinates in a neighborhood 
of $p$ such that $z_1=0$ is a local equation of $D_j,$ 
then 
\begin{equation}\label{eq.gammacycel}
 \gamma_j=\{|z_1|=\varepsilon, z_2=\dots=z_n=0\},
\end{equation}
where $\varepsilon$ is small and the orientation of $\gamma_j$ is natural, i.e.,
the differential form $\frac{1}{2\pi i} \frac{d z_1}{z_1}$ is positive on $\gamma_j.$
It is easy to see that the homological class of $\gamma_j$ in $H_1(X\setminus V)$ does not depend on the 
choice of $p,\varepsilon$  and local coordinates, this class depends only from $D_j,$ therefore $\Res_{D_j} \omega$ is well-defined.
In the local coordinates $z_1,\dots,z_n$ the differential form $\omega$ can be written as
$$\omega=\sum_{k=1}^r (Res_{D_{j_k}} \omega )\frac{d z_k}{ z_k} + \theta,$$ where $z_k=0$ is a local
equation of $D_{j_k}$ and $\theta$ is a closed holomorphic $1$-from.

\begin{theorem}\label{pr.1}
 Let $V=\bigcup^s_{j=1} D_j$ be a simple normal crossing divisor on a compact K\"{a}hler manifold $X.$
 \begin{enumerate}
   \item There is a $1$-form $\omega$ such that $\omega\in H^0(X,\Omega^1_{X}(\log V)),$ $d \omega=0$ on $X\setminus V,$ and $\Res_{D_j} \omega = a_j\in\mathbb{C}$ for $j=1,\dots,s,$   
   if and only if
 $$\sum^s_{j=1} a_j [D_j]=0$$
in $H_{2n-2}(X,\mathbb{C}).$
   \item Moreover, for any real numbers $a_1,\dots,a_s\in\mathbb{R}$ such that $\sum^s_{j=1} a_j [D_j]=0$ in $H_{2n-2}(X,\mathbb{R})$
   there exists a unique imaginary
   normalized holomorphic differential $\omega$ with logarithmic poles along $V$ such that $\Res_{D_j} \omega = a_j$ for any $j.$
 \end{enumerate}
\end{theorem}

\begin{remark}Theorem \ref{pr.1} has two meanings: first, there quite many generalized amoebas, 
because the existence conditions for normalized holomorphic differential is not too restrictive,
second, for a compact K\"{a}hler manifold $X$ and a given normal crossing divisor $V$ we can parameterize all possible generalized amoebas.
\end{remark}
\begin{proof}

The first part of this statement follows from the construction of the mixed Hodge structure on a smooth open 
quasi-projective variety, the second part follows from the first part and non-degeneracy of a period matrix of
holomorphic $1$-forms on $X.$ Now, let us give a proof.

Let us prove the "if" part of the statement $(1)$.
Let $L_j$ be a line bundle associated with the divisor $D_j,$ let us fix a hermitian metric $|\cdot|$ on $L_j.$ 
There is a holomorphic section $s_j$ such that the divisor $D_j$ is the zero set of $s_j.$
Set $\eta_j=\frac{1}{2 \pi i} \partial \log|s_j|^2.$
Suppose that $z_1=0$ is a local equation of $D_j$ in a neighborhood $U$ of a point $z\in X$, then $\eta_j$ locally looks like 
\begin{equation}\label{eq.locres}
 \eta_j=\frac{1}{2 \pi i} \frac{d z_1} {z_1}+ \theta_j,
\end{equation}
 where $\theta_j$ is a $1$-form with $C^\infty$-coefficients. Observe that $\eta_j$ is well-defined current acting on $C^\infty$-forms on $X,$ in other words,
the integral  $$[\eta_j](\psi)=\int_X \eta_j \wedge \psi$$ converges for any $C^\infty$-differential form $\psi$ on $X.$ Indeed, one can easily 
check this using partition of unity and local form (\ref{eq.locres}) of  $\eta_j.$

Observe that $d \eta_j=\overline{\partial} \eta_j = [D_j]+ \Theta_j,$ where $[D_j]$ is the currents of integration
over the divisor $D_j,$ i.e., $$[D_j](\psi)=\int_{D_j} \psi,$$ and 
$\Theta_j$ is a $(1,1)$-differential form on $X$ with $C^\infty-$coefficients. Indeed, taking the exterior derivative
of (\ref{eq.locres}) in the scene of currents we get this formula.

Consider the differential form
$$\eta=\sum^s_{j=1} a_j \eta_j,$$
then $$d \eta= \sum^s_{j=1} a_j [D_j] + \sum^s_{j=1} a_j \Theta_j.$$
By assumptions of the theorem, we have $\sum^s_{j=1} a_j [D_j]=0$ in $H_{2n-2}(X,\mathbb{C}),$ this implies that
the current of integration $\sum^s_{j=1} a_j [D_j]$ is cohomologous 
to zero, therefore it is exact and $\sum^s_{j=1} a_j [D_j]= d \mu$ for some current $\mu.$ 
Whence  $$\sum^s_{j=1} a_j \Theta_j= d (\eta-\mu)$$ 
is an exact $(1,1)$-differential form with $C^\infty-$coefficients on the  K\"{a}hler manifold $X.$ 
Using $\partial\overline{\partial}-$lemma we obtain a smooth function $f$ such that 
$$\overline{\partial}\partial f = \sum^s_{j=1} a_j \Theta_j.$$ 

Consider the differential form $$\omega = \eta - \partial f.$$ Since $d \omega = \sum^s_{j=1} a_j [D_j],$ 
we get $d \omega \equiv 0$ on $X\setminus V.$ Hence $\omega$ 
is a closed holomorphic form  on $X\setminus V.$ Moreover, by construction, $\omega$ has a 
logarithmic-singularities along $V$ and prescribed residues along $D_j, j=1,\dots,s$.
This proves "if" part of  statement $(1)$.

Let us prove the "only if" part of the statement $(1)$. Suppose there is a form $\omega$ 
such that $\omega\in H^0(X,\Omega^1_{X}(\log V)),$ $d \omega=0$ on $X\setminus V,$ and $\Res_{D_j} \omega = a_j\in\mathbb{C}$ for $j=1,\dots,s.$
Observe that  $\omega$ is a  well-defined current on $X$, indeed the integral $$
[\omega](\phi)=\int_X \omega \phi $$ converges for any $C^\infty$-forms $\phi$ on $X.$
Then, in the sense of currents, we have $d \omega = \sum^s_{j=1} a_j [D_j],$ where $[D_j]$ is the current of integration along $D_j.$
Thus $\sum^s_{j=1} a_j [D_j]$ is an exact current and $\sum^s_{j=1} a_j [D_j]=0$ 
in $H_{2n-2}(X,\mathbb{C}).$

Let us prove the "existence" part of the statement $(2)$.
Let $\widetilde{\omega}$ be a closed form from $H^0(X,\Omega^1_{X}(\log V))$ 
with residues $\Res_{D_j} \widetilde{\omega} = a_j\in \mathbb{R}$ for any $j.$ 
The periods of this form are not necessarily pure imaginary, we are going to
modify $\widetilde{\omega}$ in such a way that all periods of the modified form will be pure imaginary numbers.

Consider the sequence
\begin{equation}\label{sq.tau_i}
 \bigoplus^s_{j=1} H_{0}(D_j,\mathbb{R}) \stackrel{\tau}{\rightarrow} H_{1}(X\setminus V,\mathbb{R}) \stackrel{i_*}{\rightarrow} 
H_{1}(X,\mathbb{R})\rightarrow 0,
\end{equation}
where the map  $\tau$ takes a class of the point $p\in D_j$ to a small canonically oriented circle $\gamma_j$ around the hypersurface $D_j$  (\ref{eq.gammacycel}).
The map $i_*$ is induced by the inclusion $i: X \setminus V \hookrightarrow X.$  In general there are linear 
relations between classes of cycles $\gamma_1,\dots,\gamma_s.$
For $p\in D_j$ the following relation holds  \begin{equation}\label{eq.restau}\Res_{D_j} \omega = \frac{1}{2 \pi i}\int_{\tau (p)} \omega.\end{equation}

\begin{lemma}
The sequence (\ref{sq.tau_i}) is exact.
\end{lemma}
\begin{proof}
First, let us prove that $i_*$ is an epimorphism. Let $e$ be a one-dimensional cycle in $X.$ Because $V$ has real codimension $2,$ we can deform $e$ to a cycle $e'$ in such a way that $[e']=[e]$ in $H_{1}(X,\mathbb{R})$ and 
$e'$ does not intersect $V.$ Therefore $ i^{-1} e'$ is a cycle in $X\setminus V$ and $i_*[i^{-1} e']=[e],$ whence the map $i_*$ is an epimorphism.

Because the cycle $\gamma_j$ is a boundary of a small disc transversal to $D_j,$ the $\im \tau$ is a subspace of $\ker i_*.$ Let us prove that $ \im \tau = \ker i_*.$
Consider a cycle $\gamma\in H_{1}(X\setminus V,\mathbb{R})$ such that $i_* [\gamma]=0.$ Then there is a chain $\delta$ in $C_2(X,\mathbb{R})$ such that  $\partial \delta = \gamma.$ Without lose of generality
we may think that $\delta$ is transversal to $V$ and there is a well-defined intersection number $(\delta,D_j)$ between $\delta$ and $D_j.$
Let $U_\varepsilon D_j$ be an $\varepsilon$-neighborhood of $D_j$ for some Riemannian metric on $X.$ Consider 
the following chain $$\delta_\varepsilon= \delta \setminus \bigcup^s_{j=1} U_\varepsilon D_j.$$
Then, for sufficiently small $\varepsilon,$ we get $\partial \delta_\varepsilon= \gamma - \sum^s_{j=1} (\delta, D_j) \gamma_j.$ Whence $[\gamma]\in \im \tau.$

\end{proof}
Using this exact sequence we can obtain a split $H_{1}(X\setminus V,\mathbb{R})=\im \tau \oplus L,$ where 
$L$ is some subspace in $H_{1}(X\setminus V,\mathbb{R}).$ Then $i_*$ is an isomorphism between $L$ and  $H_{1}(X,\mathbb{R}).$
Because $X$ is a K\"{a}hler manifolds, the dimension $\dim H_{1}(X,\mathbb{R})=2l$ is even.
Let $e_1,\dots,e_{2l}$ 
be a basis of $L$, then $i_* e_1,\dots,i_* e_{2l}$ is a basis of $H_{1}(X,\mathbb{R}).$

Let $\delta$ be a cycle in $H_{1}(X\setminus V,\mathbb{R}),$ we can write it as $\delta = \gamma + e,$ where $\gamma \in \im \tau $ and $e\in L.$
If $\gamma \in \mathrm{Im} \tau$, then from (\ref{eq.restau})  and the condition  $\Res_{D_j} \widetilde{\omega} = a_j\in \mathbb{R}$
we obtain $\int_\gamma \widetilde{\omega} \in i \mathbb{R}$. Thus $\int_\delta \widetilde{\omega} \in i \mathbb{R}$ iff $\int_e \widetilde{\omega} \in i \mathbb{R}.$

If there is a holomorphic $1$-form $\psi$ on $X$ such that $$\RRe\int_{i_* e_j} \psi = \RRe \int_{e_j} \widetilde{\omega}$$ for  $j=1,\dots,2l,$
then  $\widetilde{\omega}-\psi$ is a form with pure imaginary periods and with required residues, 
indeed, the addition of a global holomorphic form to $\widetilde{\omega}$ does not change the residues.
Because $X$ is a compact K\"{a}hler manifold, the cohomology  group $H^1(X,\mathbb{C})$ has the decomposition
\begin{equation}\label{eq.Hodge1}H^1(X,\mathbb{C})= H^{1,0}(X,\mathbb{C})\oplus H^{0,1}(X,\mathbb{C}),\end{equation} 
and \begin{equation}\label{eq.Hodge2}\overline{H^{1,0}(X,\mathbb{C})} = H^{0,1}(X,\mathbb{C}),\end{equation} 
these are standard facts from Hodge theory of K\"{a}hler manifolds.
Let $\psi_1,\dots,\psi_l$ be a basis of holomorphic $1$-forms of $H^{1,0}(X,\mathbb{C}).$
From  (\ref{eq.Hodge1}) and (\ref{eq.Hodge2}), it follows that
$$\varphi_1=\frac{\psi_1+\overline{\psi}_1}{2},\dots,\varphi_l=\frac{\psi_l+\overline{\psi}_l}{2},
\varphi_{l+1}=\frac{\psi_1-\overline{\psi}_1}{2 i},\dots,\varphi_{2l}=\frac{\psi_l-\overline{\psi}_l}{2 i}$$ 
is a basis of $H^1(X,\mathbb{R}).$ 
Consider the period matrix $A=(x_{jk}+i y_{jk}),$ where
$$\int_{i_* e_k} \psi_j =x_{jk}+i y_{jk},j=1,\dots,l,k=1,\dots,2l; x_{jk},y_{jk}\in \mathbb{R}.$$
For any holomorphic form $\psi\in H^{1,0}(X,\mathbb{C})$ there is a basis decomposition
$$\psi=\sum^l_{j=1} (u_j-i v_j) \psi_j,$$ where  $u_j,v_j\in \mathbb{R}.$
We would like to find $u_j, v_j$ such that
$$\RRe \int_{i_* e_k} \psi = \sum_j u_j x_{jk} + \sum_j v_j y_{jk} = \RRe \int_{i_* e_k} \widetilde{\omega},$$
one can solve this system of linear equations if the matrix $$M=\left(
                                                          \begin{array}{c}
                                                            \RRe A \\
                                                            \im A \\
                                                          \end{array}
                                                        \right)$$
is nondegenerate. This matrix is the matrix of pairing between
$i_* e_1,\dots, i_* e_{2l}$ and $\varphi_1,\dots, \varphi_{2l},$ i.e., 
$M_{jk}=\int_{i_* e_k} \varphi_j.$ Since $i_* e_1,\dots, i_* e_{2l}$ is a basis of 
$H_1(X,\mathbb{R})$ and $\varphi_1,\dots, \varphi_{2l}$ is a basis of $H^1(X,\mathbb{R}),$ 
 the pairing between them is nondegenerate, whence $M$ is a nondegenerate matrix and
for any real values $\RRe \int_{e_1} \widetilde{\omega},\dots, \RRe \int_{e_{2l}} \widetilde{\omega}$ 
there is a holomorphic $1$-form $\psi$ such that $\RRe \int_{i_* e_k} \psi=\RRe \int_{e_k} \widetilde{\omega}.$ 
Then the form $\omega=\widetilde{\omega}-\psi$ is an imaginary normalized holomorphic differential with a given residues.

Let us show uniqueness. Suppose there are two normalized holomorphic differentials
$\omega_1, \omega_2$ such that  $\Res_{D_j} \omega_r = b_j$ for $r=1,2$ and any $j.$ 
Locally these differentials has the following form $$\omega_r=\sum_{k} b_{j_k} \frac{d z_k}{ z_k} + \theta_r,$$ where $\theta_r$ is a 
holomorphic $1$-form and $z_k=0$ is a local equation of the divisor $D_{j_k}$, therefore the singular parts of $\omega_1$ and $\omega_2$ are equal. 
Thus $\psi=\omega_1-\omega_2$ is a holomorphic $1$-form on $X.$ 

\begin{lemma}\label{lm.form}
Let $\psi$ be a closed holomorphic $p$-differential form on a compact K\"{a}hler manifold $X.$
If $\int_\gamma \psi \in \mathbb{R}$  for any cycle $\gamma\in H_p(X,\mathbb{R}),$
then 
$\psi\equiv0.$
\end{lemma}
\begin{proof}
Since $\int_\gamma \psi \in \mathbb{R}$ 
for any cycle $\gamma \in H_p(X,\mathbb{R}),$ we have $\int_\gamma\psi-\overline{\psi} = 0,$ where $\overline{\psi}$ 
is a complex conjugated form, 
which is an antiholomorphic $p$-form. Thus $\psi-\overline{\psi}$ is exact. Since $X$ is a compact K\"{a}hler 
manifold, $\psi$ and $\overline{\psi}$ are harmonic forms. On a compact K\"{a}hler manifold a harmonic 
form is exact if and only if it is identically zero, 
therefore $\psi-\overline{\psi}\equiv0.$ Because  $\psi$ and $\overline{\psi}$ have different bidegrees, we get $\psi\equiv0.$
\end{proof}

By assumptions we have $\int_\gamma \psi \in i \mathbb{R}$ 
for any cycle $\gamma \in H_1(X,\mathbb{R}).$
Thus we can apply Lemma \ref{lm.form} to the form $i \psi,$ and we get $\psi\equiv0.$ Whence $\omega_1=\omega_2.$
\end{proof}

\subsection{Asymptotic behavior of an amoeba and its asymptotic fan}

Recall that $V=\bigcup^s_{j=1} D_j.$
Let us denote $$v_j=(-\Res_{D_j} \omega_1,\dots,-\Res_{D_j} \omega_{m}).$$
We denote by $\Sigma_S$ the cone in $\mathbb{R}^m$ spanned by $v_j,j\in S\subset \{1,\dots,s\},$ i.e., 
$$\Sigma_S=\{v=\sum_{j\in S} \lambda_j v_j\in \mathbb{R}^m: \lambda_j\geq 0\}.$$

\begin{definition}
The \emph{asymptotic fan} $\Sigma_\omega$ of the generalized amoeba $\mathcal{A}_\omega $ is the set of all 
cones $\Sigma_S$ such that $\bigcap_{j\in S} D_j\neq \emptyset.$
The \emph{ support of the asymptotic fan} is the following set
 $$|\Sigma_\omega|=\bigcup_{\{S:\bigcap_{j\in S} D_j\neq \emptyset\}} \Sigma_S.$$
\end{definition}
 
\begin{remark}
We should warn that in general an asymptotic fan is not a fan in the sense of toric varieties. 
In the general case some cones may intersect each other by their interior points, but 
the standard definition of a fan requires that an intersection 
of two cones is their common face.
In the case when $m=n+1$ and $\mathcal{A}_F$ is a classical amoeba, 
the support of the corresponding asymptotic fan $|\Sigma_F|$ is the support of the $n$-skeleton of the normal fan of the Newton polytope $N_F.$ 
(The $n$-skeleton of a polyhedral complex is the set of faces of dimension at most $n$. 
See definition of the Newton polytope in the next section). 
\end{remark}

To make the next statement we need to introduce some notation. The \emph{ Hausdorff distance} between two sets $A,B\subset \mathbb{R}^m$ is equal to
$$ d_H(A,B)=\max(\sup_{a\in A}  \underset{b\in B}{\inf\vphantom{\sup}}\parallel a-b \parallel,\sup_{b\in B} \underset{a\in A}{\inf\vphantom{\sup}}\parallel a-b \parallel).$$
For a given positive real number $c$  the \emph{ $c$-neighborhood} $U_c(A)$ of a set $A\subset\mathbb{R}^m$ is the following set
$$U_c(A)=\{x\in \mathbb{R}^m: \exists y\in A \mbox{ such that } \parallel x-y \parallel<c\}.$$
\begin{proposition}\label{pr.fan} There is a constant $c>0$ such that 
$\mathcal{A}_\omega \subset U_c (|\Sigma_\omega|)$ and $|\Sigma_\omega| \subset U_c (\mathcal{A}_\omega).$ 
Consider an amoeba $\mathcal{A}_{\frac{1}{t}\omega},$ where 
${\frac{1}{t}\omega}:= ( \frac{1}{t} \omega_1 ,\dots, \frac{1}{t} \omega_m ),$ then 
$$\lim_{t\rightarrow +\infty} d_H(\mathcal{A}_{\frac{1}{t}\omega},|\Sigma_\omega|)=0.$$
\end{proposition}
\begin{proof}
Let $\{U_l\}^N_{l=1}$ be a finite cover of $X$ such that $U_l$ is a compact unit polydisc with 
local coordinates $z^l_1,\dots,z^l_n,$ i.e., $$U_l=\{(z^l_1,\dots,z^l_n)\in \mathbb{C}^n:||z^l_k||\leq 1\},$$ and $V$ is given by 
the equation ${z^l_1\cdot}\dots{\cdot z^l_{r(l)}}=0$ in $U_l$, where $z^l_k=0$ is a local equation of the divisor $D_{r(l,k)},$ here $r(l,k), r(l)$
are some numbers depending on $l$ and $k$. Because $X$ is compact and $V$ has simple normal crossings this cover exists. 
In the local coordinates we have
$$\omega_j|_{U_l}=\sum^{r(l)}_{k=1} (Res_{D_{r(l,k)}} \omega_j) \frac{d z^l_k}{z^l_k} + \varphi_{lj}(z),$$
where $\varphi_{lj}(z)$ is a closed form on $U_l$ and $\varphi_{lj}(z)$ is holomorphic in some open neighborhood of $U_l$.

Fix a set of points $\{p_1,\dots,p_N\}, p_l\in U_l\setminus V,$
and consider the function $$F_{lj}(p)=\int^p_{p_l} \varphi_{lj}(z),$$ it is bounded on $U_l.$
Recall that
$$\Log_\omega(p)=(\RRe \int^p_{p_0} \omega_1,\dots, \RRe \int^p_{p_0} \omega_m),$$
hence for $p$ from $U_l$ the $j$-th component of this map looks like
$$(\Log_\omega(p))_j=\sum^{r(l)}_{k=1} (Res_{D_{r(l,k)}} \omega_j )\log |z^l_k(p)| + \RRe F_{lj}(p)  + \RRe \int^{p_l}_{p_0} \omega_j,$$
where $z^l_k(p)$ are local coordinates of the point $p.$

Consider the vectors $a_k=(-a_{k1},\dots,-a_{km})\in \mathbb{R}^m,k=1,\dots,r,$ and define the map $\Log_{a_1,\dots,a_r}$ 
from the compact unit polydisc $$U=\{(z_1,\dots,z_n)\in \mathbb{C}^n:||z^l_k||\leq 1\},$$ to $\mathbb{R}^m:$
$$\Log_{a_1,\dots,a_r}(z_1,\dots,z_n)=(\sum^{r}_{k=1} a_{k1} \log |z_k|,\dots,\sum^{r}_{k=1} a_{km} \log |z_k|).$$
Then the image of $\Log_{a_1,\dots,a_r}$ is equal to the cone
$$\Sigma_{a_1,\dots,a_r}=\{a=\sum^r_{k=1} \lambda_k a_k\in \mathbb{R}^m: \lambda_j\geq 0\}.$$

For $p\in U_l$ we get
$$(\Log_\omega(p))_j= (\Log_{v_{r(l,1)},\dots,v_{r(l,r(l))}} (z^l(p)))_j + \RRe F_{lj}(p)  + \RRe \int^{p_l}_{p_0} \omega_j.$$
Because $\RRe F_{lj}(p)  + \RRe \int^{p_l}_{p_0} \omega_j$ is bounded on $U_l$, there is a constant $c_l>0$ such that
the image $\Log_\omega(U_l)$ of $U_l$ under the map $\Log_\omega$ is a subset of  $c_l-$neighborhood 
$U_{c_l}(\Sigma_{v_{r(l,1)},\dots,v_{r(l,r(l))}})$ and $\Sigma_{v_{r(l,1)},\dots,v_{r(l,r(l))}}$ 
is a subset of $c_l-$neighborhood $U_{c_l}(\Log_\omega(U_l)).$  
Observe that $\bigcup^N_{l=1} \Sigma_{v_{r(l,1)},\dots,v_{r(l,r(l))}}$ equals  $|\Sigma_\omega|,$
therefore $\mathcal{A}_\omega \subset U_c (\Sigma_\omega)$ and $|\Sigma_\omega| \subset U_c (\mathcal{A}_\omega)$ for $c=\max\{c_1,\dots,c_N\}.$ 
Moreover, this means that  $d_H(\mathcal{A}_{\omega},|\Sigma_\omega|)\leq c.$ 

Because
$$(\Log_{\frac{1}{t} \omega}(p))_j=\sum^{r(l)}_{k=1} (\frac{1}{t} Res_{D_{r(l,k)}} \omega_j )\log |z^l_k(p)| 
+ \RRe \frac{1}{t} F_{lj}(p)  + \RRe \frac{1}{t} \int^{p_l}_{p_0} \omega_j$$ and $\Sigma_{a_1,\dots,a_r}=\Sigma_{\frac{1}{t}a_1,\dots,\frac{1}{t}a_r}$ 
for any $t>0,$ we get $\mathcal{A}_{\frac{1}{t} \omega} \subset U_{\frac{1}{t}c} (|\Sigma_\omega|)$ 
and $|\Sigma_\omega| \subset U_{\frac{1}{t}c} (\mathcal{A}_{\frac{1}{t} \omega}).$ Thus, 
$$\lim_{t\rightarrow +\infty} d_H(\mathcal{A}_{\frac{1}{t}\omega},|\Sigma_\omega|)=0.$$

\end{proof}

\begin{remark}
This proposition describes the shape of an amoeba $\mathcal{A}_\omega$ if we look on it from the "infinity", 
one can consider $\Sigma_\omega$ as some sort of the "tropical limit" of $\mathcal{A}_{\frac{1}{t}\omega}.$ 
Unfortunately this "tropical limit" is degenerate, because it reflects only asymptotic behavior of an amoeba and forgets all 
topological structure of compact part of an amoeba. To get a good "tropical limit" one need consider simultaneous deformation 
of $\omega$ and of the pair $(X,V).$
In particular, the appropriate tropical limit of generalized amoebas of complex curves was studied in \cite{Lan}.
\end{remark}

\subsection{Nondegeneracy condition}
Let $J$ be a subset of \{1,\dots,m\} of cardinality $n$.
Consider the differential form
$$\Omega_J=\omega_{j_1}\wedge\dots\wedge\omega_{j_n},$$
where $j=\{j_1,\dots,j_n\}, j_1< \dots < j_n.$

\begin{definition}\label{df.ndc} We say that the amoeba $\mathcal{A}_\omega$ satisfies \emph{the nondegeneracy condition}
if there exists a subset $J\subset \{1,\dots,m\}, |J|=n$ such that
$\Omega_J\not \equiv 0$ on $X\setminus V.$
\end{definition}

\begin{proposition}\label{pr.nondeg}
If $\dim \Sigma_\omega: = \max_{\{S:\bigcap_{j\in S} D_j\neq \emptyset\}} \dim \Sigma_S$ is equal to $n$ then $\mathcal{A}_\omega$
satisfies the nondegeneracy condition.
Suppose that $X$ is a K\"{a}hler manifold, then $\dim \Sigma_\omega=n$ if and only if $\mathcal{A}_\omega$ satisfies the nondegeneracy condition.
\end{proposition}
\begin{proof}
Suppose that $\dim \Sigma_\omega = n.$ This means that there are irreducible divisors $D_{l_1},\dots,D_{l_n} \subset V$ 
such that $D_{l_1}\cap \dots \cap D_{l_n} \neq \emptyset$ and
the span of the vectors $v_j=(-\Res_{D_{l_j}} \omega_1,\dots,-\Res_{D_{l_j}} \omega_{m}), j=1,\dots,n$ is an $n$-dimensional space.
Hence $v_1,\dots,v_n$ are linearly independent.
Take a point $p\in D_{l_1}\cap \dots \cap D_{l_n},$ there are local coordinates $z_1,\dots, z_n$ 
in a neighborhood of $p$ such that $z_k=0$ is a 
local equation of $D_{l_k}.$ In these local coordinates the form $\omega_j$ looks like
$$\omega_j=\sum^n_{k=1} \frac{-1}{2\pi i} v_{jk} \frac{d z_k}{ z_k}+ \theta_j,$$  where $\theta_j$ is a holomorphic $1$-from 
in the neighborhood of $p$ and  $v_{jk}$ is $k$-th coordinate of the vector $v_j.$

Let $J$ be a subset of $\{1,\dots,m\}$ of cardinality $n,$ $J=\{j_1,\dots,j_n\}, j_1<\dots<j_n.$ Denote by $V_J$ the 
corresponding minor of the matrix $V=(v_{jk})_{j,k}:$
$$V_J=\det \left|
             \begin{array}{ccc}
               v_{1 j_1} & \dots & v_{1 j_n} \\
               \dots & \dots & \dots \\
               v_{n j_1} & \dots & v_{n j_n} \\
             \end{array}
           \right|.
$$
Then
\begin{equation}\label{eq.Omegaloc}
  \Omega_J= (\frac{-1}{2 \pi i})^n V_J \frac{d z_1}{z_1}\wedge \dots\wedge \frac{d z_n}{z_n}+ \Theta_J d z_1\wedge \dots\wedge  d z_n,
\end{equation}
where $\Theta_J$ is some meromorphic function such that it does not contain the term $ \frac{1}{z_1\dots z_n}$ 
in its Laurent expansion $\Theta_J=\sum_{\alpha\in \mathbb{Z}^n} C_a z^\alpha.$ Thus $\Omega_J \not\equiv0$ if $V_J \neq 0.$
By assumption $v_1,\dots,v_n$ are linear independent, whence $V_J \neq 0$ and $\Omega_J\not\equiv0.$ So, we proved the first part of the statement.

Suppose that $X$ is a K\"{a}hler manifold, let us prove the converse statement.
Any intersection of $n$ irreducible components of $V=\bigcup^s_{j=1} D_j$ is either empty or a finite number of isolated points.
Let $P(X)$ be a set of all such points, i.e.,
$p\in P(X)$ iff there is a set $J\subset \{1,\dots,s\}, |J|=n$ such that $p\in \bigcap_{j\in J} D_j.$
There are a neighborhood $U_p$ of $p\in P(X)$ and a local coordinates $z_1,\dots,z_n$ in $U_p$ such that 
$z_1\dots z_n =0$ is a local equation of $V$ on $U_p.$
We denote by $\gamma_p(X)$ the following cycle in $U_p\subset X:$
$$\gamma_p(X)=\{|z_1|=\dots=|z_n|=\varepsilon\},$$
where $\varepsilon>0$ is small.

\begin{lemma}\label{pr.reszero}
Let $X,V,P(X),$ $\gamma_p(X)$ be as above.
Let $\Omega$ be an element of $H^0(X,\Omega^n_{X}(\log V))$ 
such that 
$\int_\gamma \Omega\in \mathbb{R}$ for any $\gamma\in H_n(X\setminus V, \mathbb{R})$ 
and $\int_{\gamma_p (X)} \Omega=0$ for any $p\in P(X).$
Then $\Omega\equiv 0. $
\end{lemma}
\begin{proof}
Since $\Omega$ is a holomorphic from of  the top degree on $X\setminus V,$ the form $\Omega$ is closed.
 
Let us use the induction by $\dim X = n.$
The base of induction, $n=1$. 
Because $\Omega\in H^0(X,\Omega^1_{X}(\log V)),$ all singularities of the form $\Omega$ are first order poles.
Since $p (X)=V$ and $\int_{\gamma_p(X)} \Omega=0$ for all $p\in p (X),$ the residues of the form $\Omega$ at the points $p\in P(X)=V$ are zeros and $\Omega$ is a holomorphic form in a neighborhood of $p$.
Consequently, $\Omega$ is a holomorphic form on $X$. Using Lemma \ref{lm.form}
we get $\Omega\equiv0.$

The step of induction. 
Assume that the statement is proved for  $\dim X < n.$

Let $D_j$ be an irreducible component of the divisor $V.$ Set $V_j= \bigcup_{k\neq j}  D_k\cap D_j.$
Since the manifold $D_j$ is a submanifold of a K\"{a}hler manifold, $D_j$ is also K\"{a}hler.
Moreover, $V_j$ is a simple normal crossing divisor on $D_j$ and the set $P(D_j)$ for the divisor $V_j$ is equal to $P(X)\cap D_j.$

Let us define two maps
$$H_{n-1}(D_j\setminus V_j,\mathbb{R}) \stackrel{\delta_j}{\rightarrow} 
H_{n}( X\setminus V,\mathbb{R}),$$
and 
$$ H^{0}(X,\Omega^n_{X}(\log V) ) \stackrel{\Res_j}{\rightarrow} H^{0}(D_j,\Omega^{n-1}_{D_j}(\log V_j)).$$
Let $\gamma$ be a cycle in $D_j\setminus V_j.$ Suppose there is a fixed Riemannian metrics on $X.$ Let $\varepsilon$ be small enough and $D^2_{j,\varepsilon} \gamma$ 
be the chain consisting of all geodesic of length $\varepsilon$ starting from points of $\gamma$ and orthogonal to $D_j$, it is a $D^2-$fiber bundle over $\gamma.$
We define $\delta_j \gamma$ to be equal to $\partial D^2_{j,\varepsilon} \gamma.$

Let $U$ be a neighborhood of a point in $X$ with local coordinates $z_1,\dots, z_n$ such that $z_1\dots z_l =0$ is a local equation 
of $V$ and $z_1=0$ is a local equation of 
$D_j$. Then in the neighborhood $U$ the form $\Omega$ can be written as 
\begin{equation}\label{eq.Omegaloc2}
\Omega|_{U}=\sum_{d_1,\dots,d_n \geq -1} 
C_{d_1,\dots,d_n} z_1^{d_1}\dots z_n^{d_n} d z_1 \wedge \dots \wedge d z_n,
\end{equation}
where $C_{d_1,\dots,d_n}\in\mathbb{C}$. Then the operator $\Res_j$ is defined in the local coordinates as 
$$(\Res_j \Omega)|_{U\cap D_j} = \frac{1}{2\pi i} \sum_{d_2,\dots,d_n \geq -1} C_{-1, d_2,\dots,d_n} z_2^{d_2}\dots z_n^{d_n} d z_2 \wedge \dots \wedge d z_n.$$
One can check that $\Res_j$ is a well-defined operator that does not depend on the choice of the local coordinates. Also it is easy to check that 
the following relation holds $$\int_\gamma \Res_j \omega = \int_{\delta_j \gamma} \omega,$$
where $\omega\in  H^{0}(X,\Omega^n_{X}(\log V) )  $ and $\gamma\in H_{n-1}(D_j\setminus V_j,\mathbb{R}).$

From the definition of  $\delta_j$ it follows that 
$\gamma_p(X)$ is homologous to $\delta_j \gamma_p(D_j)$ for any $p\in P(D_j).$
Whence
$$\int_{\gamma_p(X)} \Omega =\int_{\delta_j \gamma_p(D_j)} \Omega = \int_{\gamma_p(D_j)} \Res_j \Omega =0$$
for any $p\in P(D_j).$
Also for any cycle $\gamma\in H_{n-1}( D_j\setminus V_j ,\mathbb{R}),$ we have
$$\int_\gamma \Res_j \Omega = \int_{\delta_j \gamma} \Omega \in \mathbb{R}.$$ 
Whence the form $\Res_j \Omega$ on $D_j\setminus V_j$ satisfies to the conditions of the lemma.
Thus by the inductive hypothesis we obtain  $\Res_j \Omega\equiv 0$.
Because it is true for any $D_j,$ the coefficients $C_{d_1,\dots,d_n}$ in (\ref{eq.Omegaloc2})
are equal to zero if at least one of $d_j$'s is negative. This means that $\Omega$ is a holomorphic form on $X.$
Using Lemma \ref{lm.form}
we get $\Omega\equiv0.$
\end{proof}

Suppose that $\dim \Sigma_\omega < n.$
Given a point $p\in P(X),$ consider the local representation (\ref{eq.Omegaloc}) for $\Omega_J$ in a neighborhood of $p.$
Because $\dim \Sigma_\omega < n$,  we have $V_J=0$ and, consequently, $\int_{\gamma_p(X)}\Omega_J=0.$ 
Since $\int_\gamma \omega_j\in i\mathbb{R} $ for any  
$\gamma\in H^1(X\setminus V, \mathbb{R}),$  the form $\omega_j$ is an element of $H^1(X\setminus V, i \mathbb{R})\subset H^1(X\setminus V, \mathbb{C}),$
thus we have $\Omega_J\in H^n(X\setminus V, (i)^n \mathbb{R}).$  Hence we can apply Lemma \ref{pr.reszero} to the from 
$(i)^n\Omega_J,$ whence $\Omega_J\equiv 0.$ Because this holds for any $J$ the amoeba $\mathcal{A}_\omega$ does not satisfy to the nondegeneracy condition.

\end{proof}

\begin{remark} In this article we will use the nondegeneracy condition only in the case when $m=n+1,$ 
but we believe that in the general case the nondegeneracy condition implies that  $\mathcal{A}_\omega$ behaves similarly to a classical amoeba of an algebraic 
subset of dimension $n$ in $(\mathbb{C}^*)^m.$
In the next section we will show that if $m=n+1,$ then the nondegeneracy condition is equivalent to the condition $\dim \Sigma_\omega=n$ 
(Proposition \ref{pr.nondeg2}).
\end{remark}

\section{Amoebas for m = n+1, the Ronkin function, the order map and the Newton polytope.}\label{sect.2}
In this section we will study the special case $m=n+1,$ in the classical situation it corresponds to amoebas of 
hypersurfaces in $(\mathbb{C}^*)^{m},$ which is the most studied topic in this area.

\subsection{Classical amoebas of complex hypersurfaces}
First we will recall the construction of the Ronkin function and the order map in the classical situation and basic facts related to them.
Let $F(z_1,\dots,z_{m})$ be a Laurent polynomial $$F(z_1,\dots,z_{m})=\sum_{\alpha\in A} C_\alpha z^\alpha,$$
where $\alpha$ is multi-index and  $A$ is a finite subset of $\mathbb{Z}^m,$ we assume that  $C_\alpha\neq 0$ iff  $\alpha\in A.$
The convex hull of $A$ is called the \emph{Newton polytope} of $F(z),$ we denote it by $N_F.$

We denote by $\mathcal{A}_F$ the amoeba of $X=\{F(z_1,\dots,z_{m})=0\}\subset (\mathbb{C}^*)^m.$

\begin{proposition}[\cite{GKZ}, Corollary 1.6 ] Connected components of $\mathbb{R}^m\setminus \mathcal{A}_F$ are open convex subsets.
\end{proposition}

By definition, the \emph{Ronkin function} $R_F(x)$ of $F(z)$ is given by the formula 
\begin{equation}\label{eq.Rfun}R_F(x)=\frac{1}{(2 \pi i)^m}\int_{\mathrm{Log}^{-1}(x)} \log|F(z)| \frac{d z_1}{ z_1} 
\wedge \dots \wedge \frac{d z_m}{ z_m},\end{equation}
here $$\mathrm{Log}^{-1}(x)=\{z\in (\mathbb{C}^*)^m: |z_j|=e^{x_j}\}$$ and the orientation of $\mathrm{Log}^{-1}(x)$ 
is defined by positivity of the form $\frac{1}{(2 \pi i)^m} \frac{d z_1}{ z_1} 
\wedge \dots \wedge \frac{d z_m}{ z_m}.$   Notice that this function is not necessarily smooth.

\begin{theorem}[\cite{PR}] The Ronkin function $R_F(x)$ is a convex function. If $C$ is a connected open set,
then the restriction of $R_F(x)$ to $C$ is affine linear if and only if $C$ does not intersect the amoeba $\mathcal{A}_F$.
\end{theorem}

\begin{proposition}[\cite{PR}] The restriction of the gradient vector $\nabla R_F(x)$ to a connected component $C$
of $\mathbb{R}^m\setminus \mathcal{A}_F$ is an integer valued constant function, i.e., $\nabla R_F(x)\in \mathbb{Z}^m$ if $x\in C$. 
It is equal to 
$$(\nabla R_F)_j(x)=\int_{\Log^{-1}(x)} \frac{1}{(2 \pi i)^m} \frac{z_j 
\frac{\partial F(z)}{\partial z_j }}{F(z)} \frac{d z_1}{ z_1} \wedge \dots \wedge \frac{d z_m}{ z_m},$$ 
where $x\in C $ and $(\nabla R_F)_j$ is the $j$'s coordinate of $\nabla R_F.$
\end{proposition}

Let $\Upsilon$  be a set of connected components of $\mathbb{R}^m \setminus \mathcal{A}_F.$
Let us define th \emph{order map}  $\nu_F: \Upsilon \rightarrow \mathbb{Z}^m$ as 
$$\nu_F(C)=\nabla R_F(x),$$ where $C\in \Upsilon$  and $x$ is any point from $C.$

\begin{definition} The \emph{recession cone} of a convex set $S\subset \mathbb{R}^m$ equals $$\mathrm{recc}(S)=\{v\in \mathbb{R}^m: 
\forall x\in S: x+v\in S\}.$$
\end{definition}
\begin{definition} The \emph{normal cone} to a convex set $S\subset \mathbb{R}^m$ at a point $x\in S$ 
is the cone of outer-pointing normal to the set $S$ at the point $x,$ i.e.,
$$\mathrm{norm}_S(x)=\{v\in \mathbb{R}^m: \forall y \in S: \langle y-x, v\rangle \leq 0\}.$$ 
\end{definition}

\begin{theorem}[\cite{FPT}, Theorem 2.8 ] The order map $$\nu_F: \Upsilon \rightarrow \mathbb{Z}^m$$ 
is injective, its image is contained in the Newton polytope $N_F.$  All vertices  of the Newton polytope  belongs to the image of the order map.
Therefore the number of connected components of $\mathbb{R}^m\setminus \mathcal{A}_F$ is at least equal to the number of vertices of the Newton polytope  
$N_F$ and at most equal to the total number of integer points in $N_F\cap \mathbb{Z}^m.$
\end{theorem}

\begin{proposition}[\cite{FPT}]
Let $C$ be a connected component of $\mathbb{R}^m\setminus \mathcal{A}_F.$ 
Then the recession cone of $C$ is equal to the normal cone to the Newton polytope $N_F$ at the point $\nu_F(C).$
\end{proposition}

The main aim of this section is to prove generalized amoeba analogs of the statements that are written above. 
Because there is no explicit equation of $X$ we need to find new 
definitions of the Ronkin function and the Newton polytope.

\subsection{Tropical superforms and tropical supercurrents}\label{s.supercurr} In this subsection we will develop machinery of tropical superforms. 
Tropical superforms is a new object in tropical geometry, which we believe will play crucial role in the future development of tropical geometry. For general introduction and applications of superform see
\cite{Lag}, \cite{JSS}. We will use superforms to construct the Ronkin function of the generalized amoeba.

One can consider $\mathbb{R}^m$ as a tropical analog of the complex torus $(\mathbb{C}^*)^m$.
Let $\Lambda^p T^* \mathbb{R}^m$ be the $p$-th exterior power of the cotangent bundle of $\mathbb{R}^m.$
\begin{definition}The \emph{space of tropical superforms of bidegree $(p,q)$ on $\mathbb{R}^m$} is the space of smooth sections of 
the vector bundle $\Lambda^p T^* \mathbb{R}^m \otimes \Lambda^q T^* \mathbb{R}^m.$ We denote it by $\mathcal{E}^{p,q}(\mathbb{R}^m).$
\end{definition}
Obviously, the space $\mathcal{E}^{p,q}(\mathbb{R}^m)$ is isomorphic to $$C^{\infty}(\mathbb{R}^m) \otimes \Lambda^p \mathbb{R}^m \otimes \Lambda^q \mathbb{R}^m.$$
Let $K=\{k_1,\dots,k_q\}$ be a subsets of $\{1,\dots,m\}$ of cardinality $|K|=q.$ 
We use the convention that elements of $K$ are naturally ordered: $k_1<\dots<k_q.$ Set  $d x_K=d x_{k_1}\wedge \dots \wedge d x_{k_q}.$
In terms of coordinates a superform $\omega\in \mathcal{E}^{p,q}(\mathbb{R}^m)$ can be written as
$$\omega = \sum_{\substack{|J|=p\\|K|=q}} f_{JK}(x) d x_J\otimes d x_K,$$
where the sum is taken over all subset of cardinality $p$ and $q,$ and $f_{JK}(x)\in C^{\infty}(\mathbb{R}^m).$
We should notice that $d x_J\otimes d x_K\neq 0$ even if $J\cap K \neq \emptyset.$
Let $\mathcal{E}_c^{p,q}(\mathbb{R}^m)\subset \mathcal{E}^{p,q}(\mathbb{R}^m)$ denote the subspace of tropical superform with compact support, i.e.,
the superform
$\omega = \sum f_{JK}(x) d x_J\otimes d x_K$ belongs to $\mathcal{E}_c^{p,q}(\mathbb{R}^m)$ if every $f_{JK}(x)$ has compact support.

There are two differentials
$d':\mathcal{E}^{p,q}(\mathbb{R}^m)  \rightarrow \mathcal{E}^{p+1,q}(\mathbb{R}^m),$
$d'':\mathcal{E}^{p,q}(\mathbb{R}^m)  \rightarrow \mathcal{E}^{p,q+1}(\mathbb{R}^m):$
$$d' \omega = \sum_{|I|=p, |J|=1} \sum^m_{k=1} \frac{\partial}{\partial x_k} f_{IJ}(x) d x_k\wedge d x_I \otimes d x_J$$
$$d'' \omega = (-1)^p \sum_{|I|=p, |J|=1} \sum^m_{k=1} \frac{\partial}{\partial x_k} f_{IJ}(x)  d x_I \otimes  d x_k\wedge d x_J.$$
It is easy to check that $d' d' =0,$ $d'' d''=0$ and $d' d'' = -d'' d'.$

\begin{remark}
Tropical superforms of bidegree $(p,q)$ is a tropical analog of differential forms on a complex manifold of bidegree $(p,q)$. 
The operators $d'$ and $d''$ play the role of tropical versions of the operators $\partial$ and $\overline{\partial}.$
\end{remark}

Let us fix a volume form with a constant coefficient $$\mu = c dx_1\wedge\dots \wedge d x_m = c dx, c>0,$$
then we can write an $(m,m)$-tropical superform $\omega$ as
$$\omega= f(x) \mu \otimes \mu= f c^2 dx \otimes d x.$$
The tropical integral of the superform $f(x) \mu \otimes \mu$ over $\mathbb{R}^m$ is defined to be
equal to $$\int_{\mathbb{R}^m} f(x) \mu \otimes \mu:= (-1)^{\frac{m(m-1)}{2}}\int_{\mathbb{R}^m} f(x) \mu,$$
where the right hand side is a usual integral of an $m$-form over $\mathbb{R}^m.$
This integral depends on the choice of the volume form $\mu,$ we will use the standard volume form $\mu=d x.$

Suppose $\varphi\in \mathcal{E}_c^{m,m-1}(\mathbb{R}^m),$ then $$ \int_{\mathbb{R}^m} d'' \varphi =0.$$
The same fact is true for $\varphi\in \mathcal{E}_c^{m-1,m}(\mathbb{R}^m)$ and the operator $d'.$

There is an exterior product operation 
$$\wedge:\mathcal{E}^{p,q}(\mathbb{R}^m)\times \mathcal{E}^{p',q'}(\mathbb{R}^m)\rightarrow \mathcal{E}^{p+p',q+q'}(\mathbb{R}^m).$$
It works as follows in terms of a basis. Let us denote $J,K,J',K'\subset\{1,\dots,m\},$ $|J|=p,|K|=q,|J'|=p',|K'|=q'$ 
then 
$$ x_J \otimes x_K \wedge x_{J'} \otimes x_{K'} = (-1)^{q p'} x_J \wedge x_{J'} \otimes x_K  \wedge  x_{K'}.$$
For $\omega\in \mathcal{E}^{p,q}, \varphi\in \mathcal{E}^{p',q'},$ we have
$$d'' ( \omega \wedge \varphi)=  d'' \omega \wedge \varphi +(-1)^{p+q} \omega \wedge d'' \varphi,$$
the same relation holds for the operator $d'.$
Let $\varphi\in \mathcal{E}^{q,p}(\mathbb{R}^m)$ and $ \psi\in\mathcal{E}_{c}^{m-q,m-p-1}(\mathbb{R}^m),$
then $$\int_{\mathbb{R}^m} d'' ( \varphi \wedge\psi ) = 0.$$
Therefore $$\int_{\mathbb{R}^m} d'' \varphi \wedge\psi = (-1)^{p+q+1}\int_{\mathbb{R}^m}  \varphi \wedge d'' \psi.$$
The same fact is true for the operator $d'.$

Also we define an analog of the complex conjugation 
$$I:\mathcal{E}^{p,q}(\mathbb{R}^m)\rightarrow \mathcal{E}^{q,p}(\mathbb{R}^m),$$
$$I(f(x) x_J \otimes x_K)= (-1)^{p q} f(x) x_K \otimes x_J.$$

Let $\omega= \sum \omega_{JK}(x) d x_J\otimes d x_K$ be a $(p,q)$-superform. To every compact subset $C\subset \mathbb{R}^m$ and every integer
$s \in \mathbb{Z}_{\geq 0}$, we associate a seminorm $$p^s_C(\omega)=\sup_{x\in C}\max_{\substack{|J|=p,|K|=q\\|\alpha|\leq s}}|D^\alpha \omega_{JK}(x)|,$$
where $\alpha=(\alpha_1,\dots,\alpha_m)$ runs over $\mathbb{Z}^m_{\geq 0}$ and $D^\alpha = \frac{\partial^{|\alpha|}}{\partial x_1^{\alpha_1}\dots \partial x_m^{\alpha_m}}$ is a derivation of
order $|\alpha|=\alpha_1+\dots+\alpha_m.$ We equip the space $\mathcal{E}^{p,q}(\mathbb{R}^m)$ with the topology defined by all seminorms $p^s_C$ when $s, C$ vary. The subspace
$\mathcal{E}_c^{p,q}(\mathbb{R}^m)$ is equipped with induced topology.

Let $C$ be a compact set in $\mathbb{R}^m,$ we denote by $\mathcal{E}_c^{p,q}(C)$  the subspace of elements $\omega\in \mathcal{E}^{p,q}(\mathbb{R}^m)$
with support contained in $C.$  
\begin{definition}
The space of supercurrents of bidegree $(p,q)$ on $\mathbb{R}^m$  is the space of linear forms $T$ on the space $\mathcal{E}_c^{m-p,m-q}(\mathbb{R}^m)$ 
such that the restriction of $T$ to all subspaces $\mathcal{E}_c^{m-p,m-q}(C),$  $C\Subset \mathbb{R}^m,$ is continuous.
We denote this space by $\mathcal{D}^{p,q}(\mathbb{R}^m).$ 
\end{definition}

We denote the value of a current $\varphi\in\mathcal{D}^{p,q}(\mathbb{R}^m)$ on a superfrom $\psi\in \mathcal{E}_c^{m-p,m-q}(\mathbb{R}^m)$
by $\varphi[\psi]$. We understand this formally as
$$\varphi[\psi]= \int_{\mathbb{R}^m} \varphi\wedge \psi.$$
In particular, any superform $\varphi\in \mathcal{E}^{p,q}(\mathbb{R}^m)$ defines the supercurrent
$$\varphi[\psi]:=\int_{\mathbb{R}^m} \varphi\wedge \psi.$$
 
Let $\varphi\in \mathcal{D}^{p,q}(\mathbb{R}^m),$ then the currents $d' \varphi\in \mathcal{D}^{p+1,q}(\mathbb{R}^m)$ and $d''\varphi\in \mathcal{D}^{p,q+1}(\mathbb{R}^m)$ are defined
to be equal
$$(d' \varphi)[ \psi]=(-1)^{p+q+1}\varphi[d' \psi],$$
$$(d'' \varphi)[ \psi]=(-1)^{p+q+1}\varphi[d'' \psi].$$

Any $(p,q)$-supercurrent $\varphi$ can be written as
$$\varphi=\sum_{|J|=p,|K|=q} \varphi_{JK} d x_J\otimes d x_K,$$
where $\varphi_{JK}$ are $(0,0)$-supercurrents and 
$$\varphi[\psi]=\sum_{|J|=p,|K|=q}\varphi_{JK}[(d x_J\otimes d x_K)\wedge \psi].$$

For a fixed volume form $\mu=dx$ there is an isomorphisms $\tau$ between $\mathcal{D}^{0,0}(\mathbb{R}^m)$ and $\mathcal{D}^{0}(\mathbb{R}^m),$ where $\mathcal{D}^{0}(\mathbb{R}^m)$ 
is the space of usual currents of degree $0$ on $\mathbb{R}^m.$ The isomorphisms $\tau$ is defined as
\begin{equation}\label{eq.tauisom}
\tau\varphi[f(x) d x] = (-1)^{\frac{m(m-1)}{2}} \varphi[f(x) d x\otimes \mu]  = (-1)^{\frac{m(m-1)}{2}} \varphi[f(x) d x\otimes dx],\end{equation}
where $\varphi\in \mathcal{D}^{0,0}(\mathbb{R}^m)$ and $f(x)$ is a smooth function with a compact support.

\begin{definition}
A $(p,p)$-supercurrent $$\varphi=\sum_{J,K}\varphi_{JK} dx_J \otimes dx_K$$ is called \emph{symmetric} if $\varphi[\psi]=\varphi[(-1)^{m-p}I(\psi)]$ 
for any form $\psi\in \mathcal{E}_c^{m-p,m-p}(\mathbb{R}^m)$. This condition is equivalent to the condition
$\varphi_{JK}=\varphi_{KJ}$ for all indexes $K$ and $J.$
\end{definition}
\begin{definition}We say that a symmetric $(p,p)-$current $T$ is positive if 
\begin{equation}\label{eq.postro}
 T[(-1)^{\frac{(m-p)(m-p-1)}{2}}\beta \wedge I(\beta) ] \geq 0
\end{equation}

for every $\beta\in \mathcal{E}^{m-p,0}_c(\mathbb{R}^m).$
\end{definition}
\begin{remark}
In the complex geometry any positive current is automatically symmetric. This does not hold in the tropical case. Let us consider an example.
The current $$\varphi=a_{11} dx_1\otimes d x_1 + a_{12} d x_1\otimes d x_2 + a_{21} dx_2\otimes d x_1 + a_{22} dx_2\otimes d x_2 $$ 
in $\mathbb{R}^2,$ where $a_{jk}\in \mathbb{R},$
satisfies the equation (\ref{eq.postro}) if and only if $a_{11} a_{22}-a_{12} a_{21}\geq 0.$
Obviously, this inequity may hold even if $a_{12}\neq a_{21}.$
\end{remark}

\begin{proposition}[\cite{Lag}, Proposition 1.13]\label{pr.supercurr} 
Let $\varphi$ be a $d''$-closed symmetric positive $(1,1)$-supercurrent on $\mathbb{R}^m.$ 
Then there exists a convex function $f: \mathbb{R}^m \rightarrow \mathbb{R}$ such that 
$$\varphi=d' d'' f.$$
\end{proposition}
In our opinion the proof of this proposition in \cite{Lag} is not rigorous enough. 
We will give more explicit proof, but it repeats the ideas of the proof from \cite{Lag}. 
\begin{proof} 
The current $\varphi$ can be written as $$\varphi=\sum^m_{j,k=1} \varphi_{jk} d x_j \otimes d x_k,$$ 
where $\varphi_{jk}\in \mathcal{D}^{0,0}(\mathbb{R}^m).$ 
We defined (\ref{eq.tauisom}) the isomorphism $\tau:\mathcal{D}^{0,0}(\mathbb{R}^m) \simeq \mathcal{D}^{0}(\mathbb{R}^m).$
Consider the $1$-current $$\varphi_j=\sum^m_{j=1} \tau \varphi_{jk} d x_k\in \mathcal{D}^{1}(\mathbb{R}^m),$$ the equation $d'' \varphi =0 $
is equivalent to the system of equation $d \varphi_j=0$ for $j\in\{1,\dots,m\}.$ From the Poincar\'{e} lemma for usual currents it follows that there are $0$-currents $\psi_j$ such that $d\psi_j =  \varphi_j.$
Consider the $1$-current $$\psi=\sum^m_{j=1} \psi_j d x_j.$$ Since $\varphi$ is symmetric and, consequently, $\varphi_{jk}=\varphi_{kj},$ we get 
$$d \psi=\sum^m_{j,k=1} \tau \varphi_{jk} d x_{k}\wedge  d x_{j}=\sum_{j<k} \tau \varphi_{jk} - \tau \varphi_{kj} d x_{k}\wedge  d x_{j}=0.$$ 
Therefore there is a current $f\in \mathcal{D}^{0,0}(\mathbb{R}^m)$ such that $d \tau f = \psi.$ Since $\frac{\partial^2}{\partial x_{j} \partial x_{k}} \tau f= \tau \varphi_{jk},$
we get $d' d'' f = \varphi.$

Now we will prove that $f$ can be realized by a convex function.
Let $\rho(x)$ be a nonnegative $C^{\infty}$-radial function in $\mathbb{R}^m$ such that it has compact support and $\int_{\mathbb{R}^m} \rho(x) d x = 1.$
Then consider the regularization $\tau f_\varepsilon$ of the current $\tau f:$
$$\tau f_\varepsilon(y)= \tau f[\frac{1}{\varepsilon^m}\rho(\frac{x-y}{\varepsilon}) d x]=\frac{1}{\varepsilon^m} \int_{\mathbb{R}^m}\tau f(x)  \rho(\frac{x-y}{\varepsilon}) d x,$$
where $\varepsilon > 0.$
The function $\tau f_\varepsilon(y)$ is smooth and it converges to $\tau f$ in $\mathcal{D}^{0}(\mathbb{R}^m)$ as $\varepsilon$ tends to $0.$

Let us compute $\frac{\partial^2}{\partial y_{j} \partial y_{k}} \tau f_\varepsilon(y),$ using integration by parts we get
\begin{multline}\label{eq.f_e}
\frac{\partial^2}{\partial y_{j} \partial y_{k}} \tau f_\varepsilon(y)=\int_{\mathbb{R}^m}\tau f(x)  \frac{\partial^2}{\partial y_{j} \partial y_{k}} \frac{1}{\varepsilon^m} \rho(\frac{x-y}{\varepsilon}) d x
=\\=\int_{\mathbb{R}^m}(\frac{\partial^2}{\partial x_{j} \partial x_{k}} \tau f(x))  \frac{1}{\varepsilon^m} \rho(\frac{x-y}{\varepsilon}) d x= \tau \varphi_{jk}[\frac{1}{\varepsilon^m}\rho(\frac{x-y}{\varepsilon}) d x].
\end{multline}
Let us show that $\tau f_\varepsilon(y)$ is a convex function. It is enough to check that 
$$\sum^m_{j,k=1} v_j v_k \frac{\partial^2}{\partial y_{j} \partial y_{k}} \tau f_\varepsilon(y) \geq 0$$
for any $v=(v_1,\dots,v_m)\in \mathbb{R}^m$ and any $y\in\mathbb{R}^m.$

Consider the $(m-1,0)-$superform 
$$\rho_{v}=\sum^m_{j=1} (-1)^j v_j \sqrt{\frac{1}{\varepsilon^m} \rho(\frac{x-y}{\varepsilon})} d x[j],$$
where 
$$ d x[j] = d x_1\wedge \dots \wedge d x_{j-1}\wedge d x_{j+1} \wedge \dots \wedge d x_m.$$
Because $\varphi$ is a positive current we get
$$(-1)^{\frac{(m-2)(m-1)}{2}}\varphi[\rho_{v} \wedge I(\rho_{v})]\geq 0.$$
In the other hand,
$$(-1)^{\frac{(m-2)(m-1)}{2}}\varphi[\rho_{v} \wedge I(\rho_{v})]=(-1)^{\frac{(m-2)(m-1)}{2} + m-1}\sum^m_{j,k=1} v_j v_k \varphi_{jk}[\frac{1}{\varepsilon^m} \rho(\frac{x-y}{\varepsilon}) d x \otimes d x]= 
\sum^m_{j,k=1} v_j v_k \frac{\partial^2}{\partial y_{j} \partial y_{k}} \tau f_\varepsilon(y).$$
Thus $\tau f_\varepsilon(y)$ is a convex function. 

Let us show that the sequence $\tau f_\varepsilon$ converges to some convex function as $\varepsilon$ tend to $0.$
First, let us compute $\frac{\partial}{\partial \varepsilon} \tau f_\varepsilon(y):$
$$\frac{\partial}{\partial \varepsilon} \tau f_\varepsilon(y)=
\int_{\mathbb{R}^m} - \tau f(x) (\frac{m}{\varepsilon^{m+1}}  \rho(\frac{x-y}{\varepsilon}) +\sum^m_{j=1} \frac{1}{\varepsilon^{m+2}} (x_j-y_j) \frac{\partial\rho}{\partial t_j}(\frac{x-y}{\varepsilon} )) d x
=\dots,$$
where $\frac{\partial\rho}{\partial t_j}(\cdot)$ is the derivative of $\rho(t_1,\dots,t_m)$ by the $j-$th argument.
Since
$$\frac{\partial}{\partial x_j} (\tau f(x) (x_j - y_j) \rho(\frac{x-y}{\varepsilon}))=(\frac{\partial}{\partial x_j}\tau f(x))(x_j - y_j) \rho(\frac{x-y}{\varepsilon})
+ \tau f(x)  \rho(\frac{x-y}{\varepsilon}) + 
\frac{1}{\varepsilon} \tau f(x) (x_j - y_j)  \frac{\partial\rho}{\partial t_j}(\frac{x-y}{\varepsilon} ),$$
using integration by parts we get
$$\dots=\int_{\mathbb{R}^m} \frac{1}{\varepsilon^{m+1}} \sum^m_{j=1}(\frac{\partial}{\partial x_j}\tau f(x))(x_j - y_j) \rho(\frac{x-y}{\varepsilon}) d x = \dots$$
Since $\rho$ is nonnegative radial function and has compact support, the function $$g_j(u,t)=u \rho(t_1,\dots,t_{j-1},u,t_{j+1},\dots,t_m)$$ is an odd function as a function on the argument $u$ and  has compact support,
moreover, $g_j(u,t)\leq 0$ if $u\leq 0.$
Let us consider the function 
$$G_j(t_1,\dots,t_m)=\int^{t_j}_{-\infty} g_j(u,t) d u.$$
From the properties of $g_j(u,t)$ it follows that 
 $G_j(t_1,\dots,t_m)$ has compact support and nonpositive. Moreover $$\frac{\partial}{\partial x_j} G_{j}(\frac{x-y}{\varepsilon})=\frac{1}{\varepsilon}(x_j-y_j) \rho(\frac{x-y}{\varepsilon}).$$
 Thus, we can continue our computations as follows
 $$\dots=\frac{1}{\varepsilon^m} \int_{\mathbb{R}^m} \sum^m_{j=1}(\frac{\partial}{\partial x_j}\tau f(x)) \frac{\partial}{\partial x_j} G_{j}(\frac{x-y}{\varepsilon}) d x=
 -\frac{1}{\varepsilon^m} \int_{\mathbb{R}^m} \sum^m_{j=1} (\frac{\partial^2}{\partial x_j\partial x_j}\tau f(x))G_{j}(\frac{x-y}{\varepsilon}) d x =\dots$$
 $$\dots=\frac{1}{\varepsilon^m} \int_{\mathbb{R}^m} \sum^m_{j=1} \tau \varphi_{jj}[-G_{j}(\frac{x-y}{\varepsilon}) d x]=\dots$$
 Consider the $(m-1,0)-$superform  $$T_j= \sqrt{-\frac{1}{\varepsilon^m} G_{j}(\frac{x-y}{\varepsilon})} d x[j],$$
  where  $$ d x[j] = d x_1\wedge \dots \wedge d x_{j-1}\wedge d x_{j+1} \wedge \dots \wedge d x_m.$$
 Since $\varphi$ is positive we get 
 $$(-1)^{\frac{(m-2)(m-1)}{2}}\varphi[T_j \wedge I(T_j)]\geq 0.$$
 Finally, we obtain
 $$\dots=(-1)^{\frac{(m-2)(m-1)}{2}} \sum^m_{j=1} \varphi[T_j \wedge I(T_j)]\geq0. $$
 We showed that $\frac{\partial}{\partial \varepsilon} \tau f_\varepsilon(y)\geq0$ for any $\varepsilon>0$ and any $y\in \mathbb{R}^m,$ whence 
 $\tau f_\varepsilon(y) > \tau f_{\varepsilon'}(y)$ if $\varepsilon> \varepsilon'.$
 
 Suppose there is a point $y_0$ such that $\lim_{\varepsilon \rightarrow 0} \tau f_\varepsilon(y_0) = - \infty.$
 Since $ \tau f_\varepsilon(y)$ is convex we have 
 $$  \tau f_\varepsilon(y_0 \lambda + (1-\lambda) y_1) \leq \tau \lambda f_\varepsilon(y_0)+ (1-\lambda) \tau f_\varepsilon (y_1)$$
 for any $\lambda \in (0,1)$ and any $y_1 \in \mathbb{R}^m.$ It implies that  
 $$\lim_{\varepsilon \rightarrow 0} \tau f_\varepsilon(y) = - \infty$$ for any point $y \in \mathbb{R}^m.$
 In the other hand,  $\tau f_\varepsilon$ converges in sense of currents to  $\tau f.$ 
 Thus if $\varphi \not\equiv 0$ the current $\tau f_\varepsilon(y)$ converges $ \tau f(y) \equiv - \infty.$
 If $\varphi \not\equiv 0$ we may choose $\tau f$ to be equal to $0$ identically. We obtained a contradiction, 
 therefore $\lim_{\varepsilon \rightarrow 0} \tau f_\varepsilon(y)$ should be finite for any point $y.$
 
 Thus  $\tau f_\varepsilon$ is a decreasing point-wise bounded sequence of convex functions.
 It is easy to check that it converges to a convex function $$\widetilde{f}(y)=\lim_{\varepsilon \rightarrow 0} \tau f_\varepsilon(y).$$
 Finally, we get $$d' d'' \tau^{-1} \widetilde{f}(y) = \varphi.$$ 
\end{proof}

\begin{remark}
Because $\mathbb{R}^m$ is a tropical analog of $(\mathbb{C}^*)^m,$ the last proposition is a tropical analog of the following statement 
from the complex geometry. 
\begin{proposition}For any closed and positive $(1,1)$-current $\omega$ on $(\mathbb{C}^*)^m$ there is a plurisubharmonic 
function $\rho\in\mathrm{Psh}((\mathbb{C}^*)^m)$ such that 
$$\omega=i \partial \overline{\partial} \rho.$$
\end{proposition}
One can find the local version of this 
statement in \cite[Proposition 1.19, Chapter III]{Dem}, the proof of the local version can be generalized to the  $(\mathbb{C}^*)^m$ case.   
\end{remark}

\subsection{The Ronkin functions of generalized amoebas} First, we introduce some notation. 
Consider the function on $X\setminus V$ $$x_j(p)=\RRe \int^p_{p_0} \omega_j,$$
we have $$\Log_{\omega,p_0}(p)=(x_1(p),\dots, x_m(p)).$$
Put
$$\Omega_j=\omega_1\wedge\dots\wedge\omega_{j-1}\wedge\omega_{j+1}\wedge\dots\wedge\omega_{m}.$$
Let $$ d x[j] = d x_1\wedge \dots \wedge d x_{j-1}\wedge d x_{j+1} \wedge \dots \wedge d x_m.$$
We denote by $\mathcal{E}^{d}_c(\mathbb{R}^m)$ the space of differential forms of degree $d$ with compact support on $\mathbb{R}^m.$

\begin{definition} \label{def.ronkin1}  The Ronkin function  $R_\omega$ of the amoeba $\mathcal{A}_\omega$ is a function on $\mathbb{R}^m$ such that 
the value of the current $d' d'' R_\omega$ acting on a form $\psi\in \mathcal{E}^{n,n}_c(\mathbb{R}^m),$
$$\psi=\sum^m_{j,k=1} f_{jk}(x) d x[j]\otimes d x[k],$$
which can be formally written as
$$(d' d'' R_\omega) [\psi]=\int_{\mathbb{R}^m} (d' d'' R_\omega(x)) \wedge \psi,$$ is equal to
\begin{equation}\label{eq.Rc}(d' d'' R_\omega) [\psi]=\sum^m_{j,k=1} 
\frac {(-1)^{n}}{2^n} \RRe \int_{X\setminus V} \frac{1}{(2 \pi i)^n} 
\Log^*_\omega (f_{jk}) \Omega_{j}\wedge\overline{\Omega}_{k}.\end{equation}
\end{definition}

\begin{remark}
We believe that one can think about the $(1,1)$-tropical supercurrent $d' d'' R_\omega(x)$ as current of integration over the ``tropical manifold'' $\mathcal{A}_\omega.$
The complex geometry analog is the following. 
Let  $F(z)=0$ is a holomorphic function in $(\mathbb{C}^*)^m$ such that 
its zero set  $$X_F=\{z\in (\mathbb{C}^*)^m: F(z)=0\}$$ 
is smooth and $F(z)$ has a first order zero along $X_F=\{z\in (\mathbb{C}^*)^m\},$ then
$\frac{i}{\pi} \partial \overline{\partial} \log|F| =[X_F],$ 
where $[X_F]$ is an current of integration along $X_F.$ 
\end{remark}

There are equivalent definitions of the Ronkin function in terms of usual currents.

\begin{proposition} \label{pr.ronkin_eq2}  Suppose $R_\omega$ is the Ronkin function of the amoeba $\mathcal{A}_\omega$. 
For any $\phi(x) d x$, where $\phi(x)\in \mathcal{E}^{0}_c(\mathbb{R}^m)$ and $d x = d x_1 \wedge \dots \wedge d x_m,$ holds
\begin{equation}\label{eq.Rd2} \frac{\partial^2}{\partial x_k \partial x_j}R_\omega [\phi(x) d x]=
(-1)^{\frac{m(m-1)}{2}+j+k} \frac {1}{2^n} \RRe \int_{X\setminus V} \frac{1}{(2 \pi i)^n} 
\Log^*_\omega (\phi) \Omega_{j}\wedge\overline{\Omega}_{k},\end{equation}
where   $\frac{\partial^2}{\partial x_k \partial x_j}R_\omega [\phi(x) d x]$  can be formally written as 
$$\frac{\partial^2}{\partial x_k \partial x_j}R_\omega [\phi(x) d x]= \int_{\mathbb{R}^m} (\frac{\partial^2}{\partial x_k \partial x_j}R_\omega) \phi(x) d x.$$ 
\end{proposition}
\begin{proof}
Set $$\psi=\sum^m_{j,k=1} f_{jk}(x) d x[j]\otimes d x[k]\in\mathcal{E}^{n,n}_c(\mathbb{R}^m).$$
Then we have
$$(d' d''R_\omega) [\psi]=\int d' d''R_\omega \wedge \psi = \sum^m_{j,k=1} \int_{\mathbb{R}^m} (-1)^{n + j +k} (\frac{\partial^2}{\partial x_j \partial x_k} R_\omega) f_{jk} d x \otimes d x =\dots$$
Let us transform this integral of the tropical superform to the usual integral  
$$\dots= \sum^m_{j,k=1} (-1)^{\frac{m(m-1)}{2}} (-1)^{n + j +k}\int_{\mathbb{R}^m}  (\frac{\partial^2}{\partial x_j \partial x_k} R_\omega) f_{jk} d x = \sum^m_{j,k=1} (-1)^{\frac{m(m-1)}{2}} (-1)^{n + j +k} (\frac{\partial^2}{\partial x_j \partial x_k} R_\omega)[ f_{jk} d x]$$
Combining this equality with (\ref{eq.Rc}) we obtain (\ref{eq.Rd2}).
\end{proof}

If $\mathcal{A}_\omega$ does not satisfy the nondegeneracy condition, i.e., $\Omega_j\equiv 0$ 
for any $j,$ then $\frac{\partial^2}{\partial x_k \partial x_j} R_\omega \equiv 0$
and $R_\omega$ is an affine function. In this case  $R_\omega$ does not reflect any geometry of the amoeba, 
thus this case is not interesting for us. 
From this moment and till the end of this section 
{\bf we assume that the amoeba $\mathcal{A}_\omega$ satisfies the nondegeneracy condition.}
The Ronkin function will be our main tool to study geometry of amoebas when $m=n+1.$

\begin{theorem}\label{th.ronkin} 
  The following statements hold: 
  \begin{enumerate}
  \item The Ronkin function $R_\omega$ exists and it is unique up to addition of an affine function, 
  it is a continuous convex function on $\mathbb{R}^m.$
  \item Suppose that the amoeba $\mathcal{A}_\omega$ satisfies the nondegeneracy condition, 
  i.e., there exits $j\in\{1,\dots,m\}$ such that $\Omega_j\not\equiv 0$.
  Then for any connected open set $C\subset \mathbb{R}^m$ the restriction of $R_\omega(x)$ to 
  $C$ is affine if and only if $C$ does not intersect the amoeba $\mathcal{A}_\omega$.
  \end{enumerate}
\end{theorem}

\begin{proof}
Let $$\psi=\sum^m_{j,k=1} f_{jk}(x) d x[j]\otimes d x[k]$$
be a superform with a compact support of degree $(n,n).$
We denote by $S\in\mathcal{D}^{1,1}(\mathbb{R}^m)$
the current 
$$S[\psi]=\sum^m_{j,k=1} 
\frac {(-1)^{n}}{2^n} \RRe \int_{X\setminus V} \frac{1}{(2 \pi i)^n} 
\Log^*_\omega (f_{jk}) \Omega_{j}\wedge\overline{\Omega}_{k}.$$
We want to show that there is a function $R_\omega$ with required properties such that $d' d'' R_\omega = S.$

\begin{lemma}\label{lm.positive_ronkin} 
The current $S$ is $d''$-closed, $d'$-closed, symmetric and positive.   
\end{lemma}
\begin{proof}
Let us show that $d'' S=0.$ 
By definition, $$d'' S[\psi]=-S[d'' \psi],$$
where $\psi\in \mathcal{E}^{m-1,m-2}_c(\mathbb{R}^m).$
In coordinate terms we have $$\psi=\sum_{j<k} \sum^m_{l=1} f_{ljk}(x)  d x[l]\otimes d x[j,k],$$
where $$d x[j,k] = d x_1\wedge\dots\wedge \widehat{d x_{j}}\wedge \dots \wedge \widehat{d x_{k}}\wedge \dots\wedge d x_m,$$
i.e., $d x_{j}$ and $d x_{k}$ are missing.
Then \begin{equation}\label{eq.positive_ronkin1} 
      d'' \psi= (-1)^{m-1} \sum^m_{l=1}( \sum_{j<k} (-1)^{k} \frac{\partial}{\partial x_k} f_{ljk}(x) - \sum_{j>k} (-1)^k \frac{\partial}{\partial x_k} f_{lkj}(x) ) d x[l]\otimes d x[j].
     \end{equation}
 
Let $f(x)$ be a smooth function on $\mathbb{R}^m.$
Then we have $\Log^*_{\omega}(f(x))=f(x_1(p),\dots,x_m(p)),$
where  $x_j(p)=\RRe \int^p_{p_0} \omega_j.$
The differential form $d x_k$ on $X\setminus V$ is 
equal to  $$d x_k = \frac{1}{2} (\omega_k+ \overline{\omega}_k).$$
Observe that \begin{equation}\label{eq.positive_ronkin2} d \Log^*_{\omega}(f(x)) =\frac{1}{2} \sum^m_{j=1} \Log^*_{\omega}(\frac{\partial}{\partial x_j} f(x)) \omega_j + 
\frac{1}{2} \sum^m_{j=1} \Log^*_{\omega}(\frac{\partial}{\partial x_j} f(x)) \overline{\omega}_j.\end{equation}

Consider the differential $(n,n-1)$-form 
$$\Psi=\sum^m_{l=1} \sum^m_{j<k} \Log^*_{\omega}(f_{ljk}) \Omega_{l}\wedge \overline{\Omega}_{jk},  $$
where $$\Omega_{jk} = \omega_1\wedge\dots\wedge \widehat{\omega_j}\wedge \dots \wedge \widehat{\omega_k}\wedge \dots\wedge \omega_m,$$
i.e., $ \omega_{j}$ and $ \omega_{k}$ are missing.

Combining (\ref{eq.positive_ronkin1}) and (\ref{eq.positive_ronkin2}), we get
$$S[d''\psi]= \frac {(-1)^{n}}{2^{n-1}} \RRe \int_{X\setminus V} \frac{1}{(2 \pi i)^n} d \Psi.$$
Since all functions $f_{ljk}(x)$ have compact support and, by Proposition \ref{pr.closed},  the map $\Log^*_{\omega}$ is proper, the from $\Psi$ has a compact support and 
$\int_{X\setminus V}  d \Psi = 0.$
Thus $$d''S[\psi]=-S[d''\psi]=0.$$

The proof $d'$-closedness of $S$ repeats the proof of $d''$-closedness. 

Let us check that $S$ is symmetric.
By definition it means that $$S[\phi]=S[(-1)^{m-1}I(\phi)]$$ for any 
superform $\phi=\sum^m_{j,k=1} f_{jk} d x[j] \otimes d x[k]\in  \mathcal{E}^{m-1,m-1}_c(\mathbb{R}^m).$
Observe that $$(-1)^{m-1}I(\phi)= \sum^m_{j,k=1} f_{jk} d x[k] \otimes d x[j].$$
Then \begin{multline*} 
S[f_{jk} d x[j] \otimes d x[k]]=
\frac {(-1)^{n}}{2^n} \RRe \int_{X\setminus V} \frac{1}{(2 \pi i)^n} 
\Log^*_\omega (f_{jk}) \Omega_{j}\wedge\overline{\Omega}_{k}=\dots\\
\dots =\frac {(-1)^{n}}{2^n} \int_{X\setminus V} \frac{1}{(2 \pi i)^n} 
\Log^*_\omega (f_{jk}) \frac{1}{2}(\Omega_{j}\wedge\overline{\Omega}_{k}+(-1)^n \overline{\Omega}_{j}\wedge\Omega_{k} ) = \dots
\end{multline*}
Since $n^2$ and $n$ have the same parity, we get $$\Omega_{j}\wedge\overline{\Omega}_{k}= (-1)^{n^2} \overline{\Omega}_{k}\wedge\Omega_{j}
= (-1)^{n} \overline{\Omega}_{k}\wedge\Omega_{j}.$$
Hence, we obtain 
$$\dots=\frac {(-1)^{n}}{2^n} \int_{X\setminus V} \frac{1}{(2 \pi i)^n} 
\Log^*_\omega (f_{jk}) \frac{1}{2}(\Omega_{k}\wedge\overline{\Omega}_{j}+(-1)^n \overline{\Omega}_{k}\wedge\Omega_{j})= S[f_{jk} d x[k] \otimes d x[j]].$$
Therefore $S[\phi]=S[(-1)^{m-1}I(\phi)].$

Let us check the positivity of $S.$
By definition, $S$ is positive if and only if $$S[(-1)^{\frac{n(n-1)}{2}} \varphi \wedge I(\varphi)]\geq 0$$ for any $\varphi\in \mathcal{E}^{n,0}_c(\mathbb{R}^m).$ 
In coordinate terms we have $\varphi=\sum^m_{j=1} f_{j}(x) d x[j]\otimes 1,$
and $I(\varphi)=\sum^m_{j=1} f_{j}(x)1  \otimes d x[j].$
Set $\Phi = \sum^m_{j=1} Log^*_{\omega}(f_j) \Omega_j,$
then we have $\overline{\Phi}=\sum^m_{j=1} Log^*_{\omega}(f_j) \overline{\Omega}_j.$
It is easy to check that 
$$ S[(-1)^{\frac{n(n-1)}{2}} \varphi \wedge I(\varphi)] = 
\frac {(-1)^{n}}{2^n} (-1)^{\frac{n(n-1)}{2}}  \RRe \int_{X\setminus V} \frac{1}{(2 \pi i)^n} \Phi \wedge  \overline{\Phi}.$$
Since $\Phi$ is an $(n,0)$-form, the form $(-1)^{\frac{n(n-1)}{2}+n}  \frac{1}{(2 \pi i)^n}  \Phi \wedge  \overline{\Phi}$ is positive.
Whence $$S[(-1)^{\frac{n(n-1)}{2}} \varphi \wedge I(\varphi)]\geq 0.$$
\end{proof}

It follows from Lemma \ref{lm.positive_ronkin} and Proposition \ref{pr.supercurr} that there is a convex function $R_\omega$ such that $d' d'' R_\omega = S.$
By definition, this function is the Ronkin function of the amoeba $\mathcal{A}_{\omega}.$

Let us prove the second part of the theorem. Let $C$ be a connected open set in $\mathbb{R}^m.$ 
We assume that the amoeba $\mathcal{A}_\omega$ is nondegenerate.
We want to show that the function $R_\omega|_{C}$ is affine if and only if $C \cap \mathcal{A}_\omega = \emptyset.$ 
From the Proposition \ref{pr.ronkin_eq2} it follows that
the support $\supp \frac{\partial^2}{\partial x_j \partial x_k} R_\omega$ of the current $\frac{\partial^2}{\partial x_j \partial x_k} R_\omega$ is a subset of $\mathcal{A}_\omega.$
Therefore, if $C \cap \mathcal{A}_\omega = \emptyset$, then $\frac{\partial^2}{\partial x_j \partial x_k} R_\omega|_{C}\equiv0$ and
$$R_\omega|_{C} = A_0 + x_1 A_1 + \dots + x_m A_m$$
for some constants $A_j\in \mathbb{R}.$

The nondegeneracy condition means that there is  $\Omega_j$ such that $\Omega_j\not\equiv0$ on $X.$
Suppose that $C\cap \mathcal{A}_\omega \neq \emptyset.$ Take a point $p$ in $C\cap \mathcal{A}_\omega$ 
and take an open neighborhood $U\subset C$ of $p.$
Take a function $\phi(x)\in C^\infty (\mathbb{R}^m)$ such that $\phi(x)\geq 0$ for any $x\in \mathbb{R}^m,$  
$\phi(p) > 0$ and $\mathrm{supp}\, \phi \subset U.$
From (\ref{eq.Rd2}) we have 
$$\frac{\partial^2}{\partial x_j\partial x_j} R_\omega|_{C}[\phi d x_1\wedge\dots \wedge d x_{m}]= (-1)^{\frac{m(m-1)}{2}} \frac{1}{2^n} 
\RRe \int_{X\setminus V} \frac{1}{(2 \pi i)^n} \Log^*_\omega (\phi) \Omega_{j}\wedge\overline{\Omega}_{j}.$$
The form $(-1)^{\frac{m(m-1)}{2}} \frac{1}{(2 \pi i)^n} \Omega_{j}\wedge\overline{\Omega}_{j}$
is a nonnegative form of the top degree on $X\setminus V.$ Because  $\Omega_{j}\not\equiv0 $ and $\Omega_{j}$ 
is a holomorphic form on $X\setminus V$, the zero set of $\Omega_{j}\wedge\overline{\Omega}_{j}$ has codimension $\geq 1.$ 
Since $\Log^*_\omega (\phi) \geq 0$ and $\dim \mathrm{supp}\, \Log^*_\omega (\phi) = \dim X,$ we get
$$\frac{\partial^2}{\partial x_j\partial x_j} R_\omega|_{C}[\phi d x_1\wedge\dots \wedge d x_{m}]>0.$$

In the other hand, suppose that  $R_\omega|_{C}$ is an affine function, then $\frac{\partial^2}{\partial x_j\partial x_j} R_\omega|_{C} \equiv 0,$
and 
$$\frac{\partial^2}{\partial x_j\partial x_j} R_\omega|_{C}[\phi d x_1\wedge\dots \wedge d x_{m}]=0.$$ 
Hence $R_\omega$ is not affine in a neighborhood of $p.$

\end{proof}

\begin{remark}
This construction of $R_\omega$ defines the Ronkin function uniquely up to addition of an affine function. 
In fact, in the classical case the Ronkin is also defined up to addition of an affine function, indeed, polynomials  
$F(z)=0$ and $ z^\alpha F(z)=0 $ have identical amoebas, but their Ronkin functions are different, $ R_{z^\alpha F}(x)-R_F(x)= \langle\alpha, x\rangle.$
\end{remark}

Let us show that our definition of the Ronkin function coincides with the standard definition in the classical amoeba case.
First we need to introduce some axillary constructions. 

Let us define the linear map
\begin{equation}\label{eq.tropcomplx}\Theta:\mathcal{E}^{p,q}(\mathbb{R}^m)\rightarrow \mathcal{E}^{p,q}((\mathbb{C}^*)^m)
\end{equation}
as
$$\Theta(f(x_1,\dots,x_n) d x_J\otimes d x_K) = \frac{(i)^{q}}{(2 \sqrt{\pi})^{p+q}} f(\Log(z_1,\dots,z_n)) \frac{ d z_{j_1}}{z_{j_1}}\wedge\dots\wedge\frac{ d z_{j_p}}{z_{j_p}}\wedge
\frac{ d \overline{z}_{k_1}}{\overline{z}_{k_1}}\wedge\dots\wedge\frac{ d \overline{z}_{k_q}}{\overline{z}_{k_q}},$$
where $|J|=p,|K|=q,$ and $\mathcal{E}^{p,q}((\mathbb{C}^*)^m)$ is the space of smooth $(p,q)$-forms on $(\mathbb{C}^*)^m.$
\begin{proposition}\label{pr.thete_prop}

The map $\Theta$ is an algebra homomorphism, in particular, for any two superforms $\psi\in \mathcal{E}^{p,q}(\mathbb{R}^m)$ and 
$\varphi\in \mathcal{E}^{p',q'}(\mathbb{R}^m)$
holds: $$\Theta(\psi \wedge \varphi) = \Theta(\psi) \wedge \Theta(\varphi).$$

Also, the following relations hold:
$$ \Theta (d' \psi) = \frac{1}{\sqrt{\pi}}\partial (\Theta \psi),$$
$$ \Theta ( d'' \psi) = \frac{i}{\sqrt{\pi}}\overline{\partial} ( \Theta \psi),$$ 
$$ \Theta(I \psi)  = (i)^{p+q}\overline{\Theta(\psi)},$$ 
Moreover, let $\omega\in \mathcal{E}^{m,m}(\mathbb{R}^m)$ and $U$ be a domain in $\mathbb{R}^m$. Then
$$\int_{\Log^{-1}(U)} \Theta(\omega) = \int_{U} \omega.$$
\end{proposition}
\begin{proof} 
The proof is a straightforward computation.
\end{proof}

Let us denote by $[F(z)=0]$ the integration current defined by a hypersurface $F(z)=0,$ where $F(z)$ is a Laurent polynomial in $(\mathbb{C}^*)^m.$
\begin{proposition} Let $F(z)$ be a Laurent polynomial in $(\mathbb{C}^*)^m$ and $R_F(x)$ be its Ronkin function.
Let $\psi\in\mathcal{E}_c^{n,n}(\mathbb{R}^m),$ then
$$ d' d'' R_F [\psi]= \int_{\mathbb{R}^m} d' d'' R_F \wedge \psi = \int_{(\mathbb{C}^*)^m} [F(z)=0]\wedge \Theta(\psi).$$
In particular, if $X\setminus V$ is isomorphic to $X_F=\{ z\in (\mathbb{C}^*)^m: F(z)=0\},$
and $F(z)$ has a zero of degree $1$ along $X_F$ and $\omega_j=\frac{d z_j}{z_j}|_{X_F}.$
Then $R_F$ coincides with $R_{\omega}$ up to an affine function. In other words, the definition of the
Ronkin function of a generalized amoeba is consistent with the classical definition.
\end{proposition}
\begin{proof}
First, let us compute $R_F [\varphi]$ for $\varphi=h(x) d x \otimes d x\in \mathcal{E}_c^{m,m}(\mathbb{R}^m).$
By definition, $R_F[\varphi]$ (\ref{eq.Rfun}) is equal to
 $$R_F [\varphi] = (-1)^{\frac{m(m-1)}{2}} \int_{\mathbb{R}^m} \frac{1}{(2 \pi i)^m}(\int_{\Log^{-1}(x)} 
 \log|F(z)| \frac{d z_1}{ z_1} \wedge \dots \wedge \frac{d z_m}{ z_m}) h(x) \; d x ,$$
this iterated integral can be transformed to the integral over the space $( S^1)^m \times \mathbb{R}^m,$
this space is parameterized by the coordinates $(\theta_1,\dots, \theta_m, x_1,\dots,x_m)$,
$\theta_j\in \mathbb{R} / 2 \pi \mathbb{Z} = S^1, x_j\in \mathbb{R},$ 
$$\theta_j=\arg z_j, x_j=\log |z_j|.$$
The orientation of $( S^1)^m \times \mathbb{R}^m$ is defined by the positivity of the differential 
form $\theta_1\wedge \dots \wedge \theta_m\wedge x_1\wedge\dots \wedge x_m.$
There is a natural identification between $(S^1)^m \times \mathbb{R}^m$ and $(\mathbb{C}^*)^m,$ but the orientations of
$(\mathbb{C}^*)^m$ and $(S^1)^m \times \mathbb{R}^m$ are different.
Indeed, on $(\mathbb{C}^*)^m$ we have a natural complex orientation, which is defined by the form $(\frac{i}{2})^m d z_1\wedge 
d \overline{z}_1\wedge \dots\wedge d z_m\wedge d \overline{z}_m,$ then
\begin{multline*}(\frac{i}{2})^m d z_1\wedge d \overline{z}_1\wedge \dots\wedge d z_m\wedge d \overline{z}_m=e^{2 x_1 +\dots + 2 x_m} 
d x_1\wedge d \theta_1 \wedge \dots\wedge d x_m\wedge d \theta_m=\\
 e^{2 x_1 +\dots + 2 x_m} (-1)^{\frac{m(m+1)}{2}} d \theta_1\wedge\dots \wedge d \theta_m\wedge d x_1\wedge\dots \wedge d x_m.\end{multline*}
Since $$d x_j= d \log|z_j|=\frac{1}{2} (\frac{d z_j}{ z_j} + \frac{d \overline{z}_j}{ \overline{z}_j} ),$$ we obtain
\begin{multline*}R_F [\varphi] = (-1)^{\frac{m(m+1)}{2}+\frac{m(m-1)}{2}} \frac{1}{(2 \pi i)^m} \frac{1}{(2)^m} \int_{(\mathbb{C}^*)^m}  h(\Log(z)) \log|F(z)|\frac{d z_1}{ z_1} 
\wedge \dots \wedge \frac{d z_m}{ z_m} \wedge \frac{d \overline{z}_1}{ \overline{z}_1} 
\wedge \dots \wedge \frac{d \overline{z}_m}{ \overline{z}_m}=\\
=\int_{(\mathbb{C}^*)^m} \log|F(z)| \Theta(\varphi).\end{multline*}

From Proposition \ref{pr.thete_prop}  we get 
$$d' d'' R_F[\psi]=R_F[d' d'' \psi]=\int_{(\mathbb{C}^*)^m} \log|F(z)| \Theta(d' d'' \psi)=
\int_{(\mathbb{C}^*)^m} \log|F(z)| \frac{i}{\pi} \partial \overline{\partial} \Theta(\psi).$$
The standard fact from the complex analysis is that $$\frac{i}{\pi}  \partial \overline{\partial} \log|F(z)| = [F(z)=0]$$
Finally, we get $$d' d'' R_F[\psi]= \int_{(\mathbb{C}^*)^m} \frac{i}{\pi} \partial \overline{\partial} \log|F(z)| \wedge \Theta(\psi) = 
\int_{(\mathbb{C}^*)^m} [F(z)=0]\wedge \Theta(\psi).$$

Let us prove the second part of the statement.
Let us denote 
$$\frac{d z_{[j]}}{ z_{[j]}} = \frac{d z_1}{ z_1} 
\wedge\dots  \wedge \frac{d z_{j-1}}{ z_{j-1}} \wedge \frac{d z_{j+1}}{ z_{j+1}} \wedge\dots \wedge \frac{d z_m}{ z_m}.$$
The superform $\psi$ can be written as $$\psi=\sum^m_{j,k=1} h_{jk}(x) d x_{[j]}\otimes d x_{[k]}.$$
Then $\Theta(\psi)$ equals
$$\Theta(\psi)=\frac{i^n}{2^{2n} \pi^n}\sum^m_{j,k=1} h_{jk}(\Log(z)) \frac{d z_{[j]}}{ z_{[j]}} \wedge \frac{d \overline{z}_{[j]}}{ \overline{z}_{[j]}}.$$
Observe that  $$\Omega_j= \frac{d z_{[j]}}{ z_{[j]}}|_{X_F}$$ and $\Log|_{X_F}=\Log_{\omega}.$
Therefor $$d' d'' R_F[\psi]= \int_{(\mathbb{C}^*)^m} [F(z)=0]\wedge \Theta(\psi) = 
 \frac{i^n}{2^{2n} \pi^n} \sum^m_{j,k=1} \int_{X_F} \Log^*_{\omega}(h_{jk}) \Omega_j \wedge \overline{\Omega}_k=d' d''  R_\omega[\psi].$$
\end{proof}

\begin{proposition}\label{pr.conv} Suppose the amoeba $\mathcal{A}_\omega$ satisfies the nondegeneracy condition, 
then connected components of $\mathbb{R}^m\setminus \mathcal{A}_\omega$ are open convex subsets of $\mathbb{R}^m$.
\end{proposition}
\begin{proof}
By Proposition \ref{pr.closed} $\mathcal{A}_\omega$ is closed, whence $\mathbb{R}^m\setminus \mathcal{A}_F$ is an open set.
Let $C$ be a connected component of $\mathbb{R}^m\setminus \mathcal{A}_\omega.$ By  Theorem \ref{th.ronkin} the restriction of the 
Ronkin function $R_\omega$ to $C$ is an affine function. Take a point $x_0\in C,$ then
$$\phi(x)=R_\omega(x) -  \langle\nabla R_\omega(x_0), x- x_0\rangle - R_\omega(x_0)$$ is a 
convex function, $\phi(x)\geq 0$ for any $x\in \mathbb{R}^m,$ and $\phi(x)=0$ for any $x \in C.$

Let $\hat{C}$ be a convex hull of $C.$ Since $C$ is open, $\hat{C}$ is also open. Since $\phi(x)$ is a convex function, $\phi(x)\geq 0$ and $\phi(x)\equiv 0$ on $C$, we get 
 $\phi(x)\equiv 0$ on $\hat{C},$ thus  $R_\omega$ is affine function on $\hat{C}.$ By Theorem \ref{th.ronkin} 
the restriction of $R_\omega(x)$ to $\hat{C}$ is affine linear if and only if $\hat{C}$ does not
intersect the amoeba $\mathcal{A}_\omega$, this means that $\hat{C}=C,$ i.e., $C$ is convex.
\end{proof}
\begin{proposition}\label{pr.nondeg2} Let $m$ be equal to $n+1$, then the amoeba $\mathcal{A}_\omega$ satisfies 
the nondegeneracy condition if and only if $\dim \Sigma_\omega$ is equal to $n$.
\end{proposition}
\begin{proof}
The ``only if'' part of the proposition is a part of Proposition \ref{pr.nondeg}. If $X$ is K\"{a}hler 
manifold then the ``if'' part is also a part of Proposition \ref{pr.nondeg}.
So we need to show the ``if'' part without the K\"{a}hler assumption.

Suppose that $\dim \Sigma_\omega<n$ and $\mathcal{A}_\omega$ satisfies the nondegeneracy condition. 
By Proposition \ref{pr.fan}, there is a constant $c>0$ such that $\mathcal{A}_\omega$ is a subset of 
$c$-neighborhood $U_c(|\Sigma_\omega|)$ of $|\Sigma_\omega|.$
Because $\dim \Sigma_\omega<n,$ the set $\mathbb{R}^m\setminus U_c(\Sigma_\omega) \subset \mathbb{R}^m \setminus \mathcal{A}_\omega$ is connected.
It is easy to check that 
the convex hull of $\mathbb{R}^m\setminus U_c(\Sigma_\omega) $ is equal to $\mathbb{R}^m$. Since the amoeba satisfies the nondegeneracy condition, 
from 
Proposition  \ref{pr.conv} it follows that connected components of $\mathbb{R}^m \setminus \mathcal{A}_\omega$ are convex. 
Therefor $\mathbb{R}^m \setminus \mathcal{A}_\omega = \mathbb{R}^m,$ and
$\mathcal{A}_\omega=\emptyset.$ In the other hand, by construction, we have  $\mathcal{A}_\omega\neq \emptyset$. This contradiction proves the proposition. \end{proof}

\subsection{Order map, Newton polytope and recession cones of complement to amoeba.} 
We denote the set of connected components of $\mathbb{R}^m \setminus \mathcal{A}_\omega$ by $\Upsilon.$
\begin{definition}
Let us define the map 
$$\nu_\omega:\Upsilon \rightarrow \mathbb{R}^m$$
as follows.
If $C$ is a connected component of $\mathbb{R}^m\setminus \mathcal{A}_\omega$ and $x_C$ is any point in $C$, then we define $\nu_\omega(C)$
to be equal to $\nabla R_\omega (x_C).$ Because $R_\omega$ is an affine 
function on each connected component of $\mathbb{R}^m\setminus\mathcal{A}_\omega,$ its gradient 
is constant on each connected component of $\mathbb{R}^m\setminus\mathcal{A}_\omega.$ Thus map  $\nu_\omega(C)$ is well-defined.
This map is called the \emph{order map} of the amoeba $\mathcal{A}_\omega$. 
\end{definition}

\begin{definition} The \emph{Newton polytope} of the amoeba $\mathcal{A}_\omega$ is a convex hull of the image of the 
order map $\nu_\omega$. We denote the Newton polytope by $N_\omega.$
\end{definition}

\begin{definition} The \emph{relative interior} $\mathrm{ri \;}$ C of a set $C$ is its interior within the affine hull of $C$.
The affine hull of $C$ is the smallest affine space containing $C.$
\end{definition}

\begin{definition}\label{df.grad}  If $f(x)$ is a convex function in $\mathbb{R}^m$, then by the gradient
of $f(x)$ at $x_0$ we will mean the set
$$ \nabla f(x_0)=\{y\in \mathbb{R}^m: f(x)-f(x_0)\geq\langle y,x-x_0\rangle, \forall x\in \mathbb{R}^m\}.$$
When $f(x)$ is differentiable at $x_0$, $ \nabla f(x_0)$ consists of a single point which is just
the usual gradient of $f(x)$. In the convex analysis  $ \nabla f(x)$ is also called the subdifferential of the function $f(x)$. 
\end{definition}

\begin{theorem}\label{th.rec}
The order map is injective.
Let $C$ be a connected component of $\mathbb{R}^m\setminus \mathcal{A}_\omega.$ 
Then the recession cone of $C$ is equal to the normal cone of the Newton 
polytope $N_\omega$ at the point $\nu_\omega(C).$ 
\end{theorem}
We need several technical statements to prove this theorem. 

\begin{lemma}\label{lm.2}
Let $v$ be a vector in $\mathbb{R}^m.$ Then there is a connected component $C$ of $\mathbb{R}^m\setminus \mathcal{A}_\omega$ 
such that $v\in \mathrm{recc}(C),$ $\dim \mathrm{recc}(C) = m,$
and $v\in \mathrm{norm}_{N_\omega}(\nu_\omega(C)).$
\end{lemma}
\begin{proof}

Because $\dim \Sigma_\omega=n< m$, $\mathbb{R}^m\setminus |\Sigma_\omega|$ is an open dense subset of $\mathbb{R}^m$. 
Thus there is a connected component $K$ of $\mathbb{R}^m\setminus |\Sigma_\omega|$ such that $v$ 
is contained in the closure $\overline{K}$ of $K.$
Let $\{v_j\}^\infty_{j=1}$ be a sequence of vectors in $K$ such that $||v-v_j||\rightarrow 0$ as $j\rightarrow + \infty.$

By Proposition \ref{pr.fan}, there is a connected component $C$ of $\mathbb{R}^m\setminus \mathcal{A}_\omega$ 
such that $K+x\subset C$ for some $x\in \mathbb{R}^m$ 
(we denote by $K+x$ the set $K$ shifted by a constant vector $x$). Since $K$ is a cone and $K+x\subset C$, $K$ is a subset of $\mathrm{recc} C.$ 
Because $\dim K = m $, we get $\dim \mathrm{recc} C = m.$
Therefore for any $v\in \mathbb{R}^m$ there is a connected component $C$ such that $v\in \mathrm{recc}(C)$ and $\dim \mathrm{recc}(C) = m.$

\begin{theorem}[\cite{Rock}, Theorem 8.2 and Theorem 8.3]
Let $A$ be a non-empty closed convex set in $\mathbb{R}^m.$ Then $\mathrm{recc} A$ is closed and  $\mathrm{recc} A =\mathrm{recc} (\mathrm{ri} A).$
\end{theorem}
 
Let us apply this theorem to the closure $\overline{C}$ of the set $C.$ Since $\mathrm{ri} \overline{C} = C$, the recession cone 
of $C$  is closed. 
Because $v_j\rightarrow v$ as $j\rightarrow +\infty,$ we obtain $v\in \mathrm{recc} (C).$ We proved that for any $v\in\mathbb{R}^m$ there is a connected component 
$C$ in $\mathbb{R}^m\setminus \mathcal{A}_\omega$  such that $v\in \mathrm{recc}(C)$ and $\dim \mathrm{recc}(C) = m.$

Let us prove that $v\in \mathrm{norm}_{N_\omega}(\nu_\omega(C)).$
Let $u$ be a vector in $ K.$ Observe that for any point $y\in \mathbb{R}^m$ there is a constant $t_y\geq 0$ such that for any $t\geq t_y$ holds 
$t u+  y \in K+x\subset C.$ Consider the function $F_{y,u}(t)=R_\omega(t u + y),$ this function is convex, 
hence $\frac{\partial F_{y,u}}{\partial t}(t_1)\geq \frac{\partial F_{y,u}}{\partial t}(t_0)$ for $t_1\geq t_0.$

Because, for $t\geq t_y,$ the ray $t u + y$ does not intersect $\mathcal{A}_\omega$, the function $F_{y,u}(t)$ is affine function on $[t_y, +\infty).$
Thus $\frac{\partial F_{y,u}}{\partial t}(t_y)\geq \frac{\partial F_{y,u}}{\partial t}(t)$ for any $t.$
Observe that 
$\frac{\partial F_{y,u}}{\partial t}(t_0) = \langle\nabla R_\omega (t_0 u + y), u \rangle= \nabla_u R_\omega (t_0 u + y).$
Whence $\frac{\partial F_{y,u}}{\partial t}|_{t=t_y} = \langle \nu_\omega(C), u \rangle.$

Let $C'$ be a connected component of $\mathbb{R}^m\setminus \mathcal{A}_\omega$ and $y$ be a point in $ C'.$ Then 
$$\langle \nu_\omega(C'), u \rangle =  \frac{\partial F_{y,u}}{\partial t}(0) 
\leq \frac{\partial F_{y,u}}{\partial t}(t_y) = \langle \nu_\omega(C), u \rangle. $$
Put $u=v_j,$ then for any $j$ and any connected component $C'$ we get $$\langle \nu_\omega(C'), v_j \rangle \leq \langle \nu_\omega(C), v_j \rangle.$$ 
Since  $v_j\rightarrow v$ as $j\rightarrow +\infty,$ we get 
$$\langle \nu_\omega(C'), v \rangle \leq \langle \nu_\omega(C), v \rangle.$$
The last inequity is equivalent to $$\langle \nu_\omega(C')- \nu_\omega(C), v \rangle \leq 0.$$
Because $N_\omega$ is a convex hull of $\nu_\omega(C'), C'\in \Upsilon,$ we obtain $v\in \mathrm{norm}_{N_\omega}(\nu_\omega(C)).$
\end{proof}

\begin{cor}
The Newton polytope $N_\omega$ is bounded.  
\end{cor}
\begin{proof}
Consider a set of vectors $v_1,\dots,v_{m+1}\in\mathbb{R}^m$ such that  $v_1,\dots,v_m$ is an orthonormal basis and  $v_{m+1}= -v_1-\dots-v_m.$
Then by Lemma \ref{lm.2} there is a connected component $C_j$ of $\mathbb{R}^m\setminus \mathcal{A}_\omega$ 
such that $v_j\in  \mathrm{norm}_{N_\omega}(\nu_\omega(C_j))$ for any $j.$
This means that $$N_\omega\subset \bigcap^{m+1}_{j=1}\{x\in \mathbb{R}^m: \langle x, v_j \rangle \leq d_j \}$$
for some numbers $d_j.$  The latter set is bounded for any choice of $d_j.$
\end{proof}

\begin{lemma}\label{lm.1}  Let $f(x)$ be a convex function on $\mathbb{R}^m.$
Given two open convex sets $C$ and $C'$ such that the restrictions $f(x)|_{C}$ and $f(x)|_{C'}$ are affine functions.
Suppose there is a vector $v$ such that $v$ belongs to recession cones of sets $C$ and $C',$ i.e., $v\in \mathrm{recc}(C)$ and $v\in\mathrm{recc}(C').$
Then $\nabla_v f(x)=\nabla_v f(x')$ for any $x\in C$ and $x'\in C',$ where  $\nabla_v f(x)$ is the directional derivative of $f(x)$  in  the direction $v.$ 
\end{lemma}
\begin{proof}
Without lose of generality we may assume that $f(x)|_{C} \equiv 0$ and $f(x)|_{C'}  = \langle a, x \rangle + a_0,$ where 
$a$ is a vector from $\mathbb{R}^m$ and $a_0$ is a real number.
Because $v\in \mathrm{recc}(C) \cap \mathrm{recc}(C'),$ we can find two points $x\in C,$ $x' \in C'$ and a real number $d$ such that $x$ and $x'$ belong 
to the plane $\{y\in \mathbb{R}^m: \langle v, y \rangle = d\}.$ 
Let us denote $x_t= x + t v$ and $x'_t= x' + t v.$ Obviously, for any $t>0,$ we have $x_t\in C$ and $x'_t\in C'.$
Consider the function $F_t(\lambda) = f(\lambda x_t + (1-\lambda) x'_t).$ Since $f$ is convex, $F_t(\lambda)$ is also convex.
We have $$\frac{\partial}{\partial \lambda} F_t(\lambda) = 
\langle\nabla f(y) |_{y=\lambda x_t + (1-\lambda) x'_t}, x_t - x'_t\rangle = \langle\nabla f(y) |_{y=\lambda x_t + (1-\lambda) x'_t}, x - x'\rangle.$$
Because  $F_t(\lambda)$ is convex, we have 
$$F_t(\lambda)\geq \frac{\partial}{\partial \lambda} F_t(\lambda)|_{\lambda = \lambda_0} (\lambda - \lambda_0) + F_t(\lambda_0)$$
for any $\lambda$ and $\lambda_0.$
Thus, for $\lambda_0 = 0,$ we get $$F_t(\lambda) \geq \langle a, x\rangle + a_0 + t \langle a, v\rangle,$$
and, for $\lambda_0=1,$ we get $$F_t(\lambda) \geq 0.$$
This gives us two inequities $$F_t(1)=0 \geq \langle a, x\rangle + a_0 + t \langle a, v\rangle$$  and
$$F_t(0)= \langle a, x'\rangle + a_0 + t \langle a, v\rangle \geq 0,$$
they should hold for any $t \geq 0.$ It is possible if and only if $ \langle a, v \rangle =0.$ Because 
$\nabla_v f(y)|_{y=x} =0$ and $\nabla_v f(y)|_{y=x'} =\langle a, v \rangle,$ this gives us the statement of the lemma.
\end{proof}

\begin{proposition}\label{pr.grad} The following inclusions hold $$\mathrm{ri\;} N_\omega \subset \im \nabla R_\omega \subset N_\omega.$$
\end{proposition}
\begin{proof}

Proof of this proposition repeats the proof of Theorem 3 in \cite{PR}.
The image of the gradient map $\im \nabla R_\omega$ of the convex function $R_\omega$ by Definition \ref{df.grad} is equal to the 
set of $\xi\in \mathbb{R}^m$ such that $R_\omega(x) - \langle \xi , x\rangle$ attains its global
minimum in $\mathbb{R}^m$. Let $G$ be the set of all $\xi\in\mathbb{R}^m$ such that $R_\omega (x) - \langle \xi , x\rangle$ 
is bounded from below on $\mathbb{R}^m$.
It is easy to see that $\im \nabla R_\omega$ is contained in $G$ and that the interior of $G$ is contained in $\im \nabla R_\omega$. 
Hence the statement will be proved if we prove that $G$ is equal to the Newton polytope $N_\omega$ of $\mathcal{A}_\omega$.

Let $\Upsilon$ be a set of connected components of $\mathbb{R}^m\setminus \mathcal{A}_\omega.$ For each $C\in \Upsilon$ 
choose a point $x_C$ such that $x_C\in C.$
Consider the function $$S(x)=\sup_{C\in \Upsilon } \langle \nu_\omega(C) ,x-x_C \rangle + R_\omega (x_C).$$
Obviously, $R_\omega(x)\geq S(x)$ for any $x\in \mathbb{R}^m.$
Suppose $\xi$ is in $N_\omega,$ then the function $S(x) - \langle \xi , x\rangle$ is bounded from below. Since 
$R_\omega (x) - \langle \xi , x\rangle \geq S(x) - \langle \xi , x\rangle,$ $R_\omega (x) - \langle \xi , x\rangle $ is also bounded from below. Therefore, if $\xi\in N_\omega,$ then $\xi\in G.$

Suppose $\xi$ is outside $N_\omega$. Take $y\in \mathrm{R}^m$ such that $\langle \xi , y \rangle > \sup_{\zeta\in N_\omega} \langle \zeta , y \rangle$ 
(because $N_\omega$ is convex, this $y$ exists). 
By Lemma \ref{lm.2}, there is a connected component $C\in \Upsilon$ such that 
 $y\in \mathrm{recc}(C)$ and $y\in \mathrm{norm}_{N_\omega}(\nu_\omega(C)),$ this means that the supremum
 $\sup_{\zeta\in N_\omega} \langle \zeta , y \rangle$
 is attained at the point $\zeta=\nu_\omega(C).$
Because $y\in \mathrm{recc}(C),$ for any $t>0$ we have $x=x_C + t y \in C$ and 
$R_\omega(x_C + t y) = R_\omega(x_C) + \langle\zeta , t y \rangle .$ 
Whence $$R_\omega(x_C + t y) - \langle \xi , x_C + t y \rangle =  R_\omega(x_C) + \langle \zeta  - \xi ,  t y \rangle - \langle \xi, x_C \rangle\rightarrow - \infty $$  
as $t \rightarrow \infty.$ Thus, by definition, $\xi$ is not in $G.$ Therefore, if $\xi\not\in N_\omega,$ then $\xi\not\in G.$ Whence we get $N_\omega = G.$
\end{proof}

Now we can prove Theorem \ref{th.rec}.
\begin{proof}

Let us show that the order map is injective.
Suppose there are two connected components $C$ and $C'$ such that $\nu_\omega(C)=\nu_\omega(C').$ 
Without lose of generality we may assume that  $\nu_\omega(C)=\nu_\omega(C')=0.$
By Theorem \ref{th.ronkin}, $R_\omega$ is convex, therefore for any $y\in \mathbb{R}^m$ holds $$R_\omega(y)\geq \langle \nabla R_\omega(x), y-x\rangle + R_\omega(x).$$
Whence we get $R_\omega(y)\geq R_\omega(x)$ 
and  $R_\omega(y)\geq R_\omega(x'),$ where $x\in C$ and $x'\in C'.$ 
This implies $R_\omega(x) = R_\omega(x').$ We may assume that $R_\omega(x) = R_\omega(x')=0.$
Thus we have $R_\omega(y) \geq 0$ for any $y\in \mathbb{R}^m.$ In the other hand, since $R_\omega$ is 
convex, for any $\lambda\in [0,1]$ holds $$R_\omega(\lambda x  + (1-\lambda) x') \leq 0.$$
Whence $R_\omega(y)\equiv 0$ on a convex hull $\mathrm{conv}(C\cup C')$ of $C\cup C'.$ 
By Proposition \ref{pr.conv}, $C$ and $C'$ are open, hence $\mathrm{conv}(C\cup C')$ is also open. 
By Theorem \ref{th.ronkin}, the restriction of the Ronkin function  $R_\omega(y)$ 
to a connected open set $U$ is affine
if and only if $U$ does not intersect $\mathcal{A}_\omega.$ 
Whence $\mathrm{conv}(C\cup C') \subset \mathbb{R}^m\setminus \mathcal{A}_\omega.$ Since $\mathrm{conv}(C\cup C')$ is connected, the connected component $C$ coincides with  $C'.$

Let us show that $v\in \mathrm{recc}(C)$  implies $v\in \mathrm{norm}_{N_\omega}(\nu_\omega(C)).$ 
Suppose that $v\in \mathrm{recc}(C).$ 
By  Lemma \ref{lm.2}, there is a connected component $C'$ 
such that $v\in \mathrm{norm}_{N_\omega}(\nu_\omega(C'))$ and $v\in \mathrm{recc}(C').$ 
By Lemma \ref{lm.1}, we get $\nabla_v R_\omega(x)= \nabla_v R_\omega(x'),$ where $x\in C,x'\in C'.$
Hence $ \langle \nu_\omega (C') , v\rangle = \langle \nu_\omega (C) , v\rangle.$

Consider the function $f_v(y)=\langle y, v \rangle$ on the set $N_\omega.$ 
A point $x$ is a maximum point of $f_v$ if and only if $v\in\mathrm{norm}_{N_\omega}(x).$ Because  
$v\in \mathrm{norm}_{N_\omega}(\nu_\omega(C')),$ the function $f_v$ 
attains its maximum at the point $\nu_\omega (C').$
Since $$f_v(\nu_\omega (C'))=\langle \nu_\omega (C') , v\rangle = \langle \nu_\omega (C) , v\rangle=f_v(\nu_\omega (C)),$$  
we see that $f_v$ has a maximum at the point $\nu_\omega (C).$ Therefore, $v\in  \mathrm{norm}_{N_\omega}(\nu_\omega(C)).$

Let us show that $v\in \mathrm{norm}_{N_\omega}(\nu_\omega(C))$ implies  $v\in \mathrm{recc}(C).$ 
Suppose that $v\in \mathrm{norm}_{N_\omega}(\nu_\omega(C))$. We may assume that $R_\omega|_C \equiv 0.$ 
Because $v\in \mathrm{norm}_{N_\omega}(\nu_\omega(C)),$ the function $f_v(y)=\langle y, v \rangle$ 
on the set $N_\omega$ attains its maximum at the point $\nu_\omega (C).$ 
Since, by Proposition \ref{pr.grad}, $\im \nabla R_\omega \subset N_\omega,$ 
we get $$\max_{x\in \mathbb{R}^m }\nabla_v R_\omega(x) =  \max_{x\in \mathbb{R}^m } \langle v, \nabla R_\omega(x)\rangle = \langle \nu_\omega (C), v\rangle.$$

Let us consider the function $g_x(t)=R_\omega(x+ t v).$ It is a convex function on $\mathbb{R}$, 
therefore $$\frac{\partial g_x(t)}{\partial t}|_{t=t_1} \geq \frac{\partial g_x(t)}{\partial t}|_{t=t_0},$$ 
when $t_1 > t_0.$ Because $\frac{\partial g_x(t)}{\partial t}= \nabla_v R_\omega(x+v t),$ 
we obtain $\nabla_v R_\omega(x+v t) \geq \nabla_v R_\omega(x)$ for any $x\in C$ and $t>0.$
In the other hand, since $\max_{x\in \mathbb{R}^m }\nabla_v R_\omega(x) = \langle \nu_\omega (C), v\rangle,$
we have $\nabla_v R_\omega(y) \leq \nabla_v R_\omega(x)$ for any $y\in \mathbb{R}^m, x\in C.$
Thus $\nabla_v R_\omega(x+v t) = \nabla_v R_\omega(x)$ for any $x\in C$ and $t>0.$

By assumption $R_\omega|_C \equiv 0,$ whence $\nabla_v R_\omega (x+v t) \equiv 0$ and, 
consequently,  $R_\omega (x+v t) \equiv 0$ for any  $x\in C$ and $t>0.$ 
It follows form Proposition \ref{pr.conv} that the set $$\widetilde{C}=\{x\in \mathbb{R}^m: x= y + v t, y\in C, t\geq 0\}$$
is open, also it is connected.  By construction $R_\omega|_{\widetilde{C}}\equiv 0.$ 
Therefore by Theorem \ref{th.ronkin} we get $\widetilde{C} \subset \mathbb{R}^m\setminus \mathcal{A}_\omega,$ 
whence $\widetilde{C}  = C.$ Thus $C$ contains the ray $x + v t, t\geq 0,$ this means $v\in \mathrm{recc}(C).$
\end{proof}

\begin{remark}
In the classical case the order map $\nu_F$ is integer-valued, therefore the number of connected component 
is  not greater than $|\mathbb{Z}^m \cap N_F|.$ In the generalized case
it is not possible to use this kind of argument at least in the general case. So there is a problem: find
an upper bound of the number of connected components of $\mathbb{R}^m\setminus \mathcal{A}_\omega.$
We believe that it is possible to estimate this number in terms of topology of $X\setminus V.$ 
In particular, let $F(z_1,\dots,z_m)$ be a Laurent polynomial, suppose that $X\setminus V=\{z\in (\mathbb{C}^*)^m| F(z_1,\dots,z_m)=0\}$ is smooth. 
Then, from \cite{DK}, 
it follows that the number of integer point in the Newton polytope $|N_F\cap \mathbb{Z}^m|$ is equal to $\dim F^n H^n(X\setminus V,\mathbb{C})+1,$ 
where $F^n H^n(X\setminus V,\mathbb{C})$ is 
the $n$-th term of the Hodge filtration on the cohomology group $H^n(X\setminus V,\mathbb{C}).$ The results from tropical geometry \cite{IKMZ} suggest 
that the same estimation could be true in the generalized case. 
\end{remark}

\begin{remark}
Let us notice that initial data of the generalized amoeba: a complex manifold $X$ and a set of special 
meromorphic forms $\omega_1,\dots,\omega_m$ are objects of algebraic geometry nature. Moreover, Theorem \ref{pr.1} gives a nice description of this type of data in the K\"{a}hler manifold case. 
But, so far, we have only analytical description of the Newton polytope $N_\omega.$  
It would be interesting to find a purely algebraic geometry style description of $N_\omega.$  

In the classical case the Newton polytope  $N_F$ has a straightforward description in terms of the polynomial $F(z).$ 
One can find more elaborated decryption of the Newton polytope in terms of an appropriate toric compactification $X_F$ of $\{F(x)=0\} 
\subset (\mathbb{C}^*)^m \subset  X_F.$
Then Newton polytope plays role of generalized degree of the variety $\overline{\{F(x)=0\}}$ in $X_F.$ 
The further generalization of this idea is called the Newton-Okunkov body \cite{KK}.
We think that it could be interesting to find a relation ( if there is any) between Newton-Okunkov bodies and Newton polytopes of generalized amoebas.  

\end{remark}
\subsection{Monge-Amp\`{e}re measure of the Ronkin function}
Now we are going to prove some statements about the Monge-Amp\`{e}re measure of the Ronkin function of a generalized amoeba. 
These statements are generalizations of Theorem 3 and Theorem 4 from \cite{PR}.
The proofs of these generalizations repeat the proofs of the classical amoeba versions. 

If $f(x)$ is a smooth convex function, its Hessian $\mathrm{Hess}(f)$ is a positive semi-definite matrix.
In particular, the determinant of the Hessian is a nonnegative function. The product of $\det \mathrm{Hess}(f)$ with ordinary 
Lebesgue measure is known as the real Monge-Amp\`{e}re measure of $f(x)$. We denote it by $$M f= \det \mathrm{Hess}(f) \; d x .$$
In fact, the Monge-Amp\`{e}re operator 
can be extended to all convex functions. In general, if $f(x)$ is a non-smooth convex function, $M f$ will be a positive measure, see \cite{RT}.
Hence $M R_\omega$ is a positive measure supported on $\mathcal{A}_\omega$.

Using the polarization formula for the determinant we obtain a unique symmetric multilinear form
\begin{equation}\label{eq.MApol}
M(f_1,\dots,f_m)=\sum_{J,K\in S_m} \frac{\sign J+\sign K}{m!}  \prod^m_{p=1} \frac{\partial^2}{\partial x_{j_p} \partial x_{k_p}}  f_{p} d x
\end{equation}
such that $M(f,\dots,f)=M(f),$ here $S_m$ is the symmetric group of degree $m,$ $K=(k_1,\dots,k_m), J=(j_1,\dots,j_m),$ 
and  $\sign J$ is the parity of $J.$
One can extend this multilinear form to the set of all convex functions on $\mathbb{R}^m.$

Let $E$ be a Borel set in $\mathbb{R}^m.$ Denote by $1_E$ the indicator function of the set $E,$ i.e.,
$1_E(x)=1$ if $x\in E,$ and $1_E(x)=0$ if $x\not\in E.$
Then for any $C^2-$functions $f_2,\dots,f_{m}$ we can define 
$M(R_\omega,f_2,\dots,f_m)(E)$ to be equal to
$$M(R_\omega,f_2,\dots,f_m)(E) = \sum_{J,K\in S_m} \frac{\sign J+\sign K}{m!}  
\frac{\partial^2}{\partial x_{j_1} \partial x_{k_1}} R_\omega [1_E \prod^{m}_{p=2} \frac{\partial^2}{\partial x_{j_p} \partial x_{k_p}}  f_{p} dx].$$
It is easy to check that this formula is consistent with definition of $M(f_1,\dots,f_m)$ for convex functions $f_1,\dots,f_m.$

\begin{proposition} Let $E$ be any Borel set in $\mathbb{R}^m$.
Let us denote $d^c=\partial-\overline{\partial}.$
Then \begin{equation}\label{eq.MAfun} m! M(R_\omega,f_2,\dots,f_m)(E) = 
\RRe\frac{1}{(2 \pi i)^n}\int_{\Log^{-1}_\omega(E)}  d d^c \Log^{*}_\omega f_2\wedge \dots \wedge d d^c \Log^{*}_\omega f_m.\end{equation}

Denote by $\Delta=\frac{\partial^2}{\partial x_1^2}+\dots+\frac{\partial^2}{\partial x_m^2}$ the Laplace operator. Then
$$\int_E \Delta R_\omega d x=\frac{1}{n!} \RRe \frac{1}{(2 \pi i)^n} \int_{\Log^{-1}_\omega(E)} (-\frac{1}{2} \sum^m_{j=1} 
\omega_j \wedge \overline{\omega}_j)^n.$$
\end{proposition}
\begin{proof}
Recall that $\Log_\omega(z)=(x_1(z),\dots,x_m(z)),z\in X\setminus V,$ and $d x_j=\frac{1}{2}(\omega_j+ \overline{\omega}_j).$ Therefore 
we obtain 
$$d d^c \Log^{*}_\omega f_p=-\frac{1}{2} \sum^m_{j,k=1} \Log^{*}_\omega (\frac{\partial^2}{\partial_j\partial_k} f_p) 
\omega_j\wedge \overline{\omega}_k,$$
 and

\begin{multline*}
d d^c \Log^{*}_\omega f_2\wedge \dots \wedge d d^c \Log^{*}_\omega f_m
=\\
=(-\frac{1}{2})^n (-1)^{\frac{(n-1)n}{2}} \sum^m_{j_2,\dots,j_m,k_2,\dots,k_m=1} \Log^{*}_\omega (\prod^m_{p=2}
\frac{\partial^2}{\partial_{j_p}\partial_{k_p}} f_p) 
\omega_{j_2}\wedge\dots\wedge\omega_{j_m} \wedge \overline{\omega}_{k_2}\wedge\dots\wedge\overline{\omega}_{k_2}=\\
=(-\frac{1}{2})^n (-1)^{\frac{(n-1)n}{2}} \sum_{K,J\in S_n}(-1)^{j_1+k_1} (\sign K + \sign J)  
\Log^{*}_\omega (\prod^m_{p=2}\frac{\partial^2}{\partial_{j_p}\partial_{k_p}} f_p)  \Omega_{j_1}
\wedge \overline{\Omega}_{k_1},
\end{multline*}
where $K=(k_1,\dots,k_m),J=(j_1,\dots,j_m).$

Combining the last expression with the formula for $\frac{\partial^2}{\partial x_k \partial x_j}R_\omega$ (\ref{eq.Rd2}) and for $M(R_\omega,f_1,\dots,f_m)(E)$ we obtain 
$$\RRe\int_{\Log^{-1}_\omega(E)} \frac{1}{(2 \pi i)^n} d d^c \Log^{*}_\omega f_2\wedge \dots \wedge d d^c \Log^{*}_\omega f_m=$$ 
$$= \sum_{K,J\in S_n} (\sign K + \sign J) 
\frac{\partial^2}{\partial x_{k_1} \partial x_{j_1}}R_\omega [1_E \prod^m_{p=2}\frac{\partial^2}{\partial_{j_p}\partial_{k_p}} f_p d x]= 
m! M(R_\omega,f_1,\dots,f_m)(E).$$

Consider the function
$||x||^2=x_1^2+\dots+ x_m^2.$ Since 
$\frac{\partial^2}{\partial x_{k} \partial x_{j}} ||x||^2=2 \delta_{ij}$ (Kronecker delta), using (\ref{eq.MApol}) we obtain 
\begin{equation}\label{eq.laplace}
M(R_\omega,||x||^2,\dots,||x||^2)(E)=\int_E \frac{2^{n}}{m} \Delta R_\omega d x.
\end{equation}

Let us compute $d d^c \Log^{*}_\omega  ||x||^2.$
Since  
$d x_j^2 = 2 x_j d x_j,$ we get $$d  \Log^{*}_\omega x_j^2 = \Log^{*}_\omega d x_j^2 = 2 \Log^{*}_\omega (x_j) \frac{1}{2}(\omega_j+ \overline{\omega}_j).$$
Whence $$d d^c \Log^{*}_\omega (x_j^2) = - d^c d \Log^{*}_\omega(x_j^2) = - 2 d^c \Log^{*}_\omega (x_j)\frac{1}{2}(\omega_j+ \overline{\omega}_j)= 
- \frac{1}{2}(\omega_j- \overline{\omega}_j) \wedge (\omega_j+ \overline{\omega}_j) =-\omega_j \wedge \overline{\omega}_j.$$
Thus we have $$d d^c \Log^{*}_\omega ||x||^2 = - \sum^m_{j=1} \omega_j \wedge \overline{\omega}_j.$$
Using (\ref{eq.laplace}) and (\ref{eq.MAfun}) we get 
$$\int_E \frac{2^{n}}{m} \Delta R_\omega d x=M(R_\omega,||x||^2,\dots,||x||^2)(E) = \frac{1}{m!}\RRe \frac{1}{(2 \pi i)^n} \int_{\Log^{-1}_\omega(E)} 
(-1)^n (\sum^m_{j=1} \omega_j \wedge \overline{\omega}_j)^n.$$
Whence we obtain $$\int_E \Delta R_\omega d x=\frac{1}{n!} \RRe \frac{1}{(2 \pi i)^n} \int_{\Log^{-1}_\omega(E)} (-\frac{1}{2} \sum^m_{j=1} \omega_j 
\wedge \overline{\omega}_j)^n.$$
\end{proof}

\begin{proposition}\label{pr.totalmass} The total mass of $M R_\omega$ is equal
to the volume of the Newton polytope $N_\omega$. 
\end{proposition}
\begin{proof}
From the results of \cite[Proposition 3.4 and Definition 2.6]{RT}, it is known that $$M f (E) = \mbox{Lebesgue measure of } \nabla f (E),$$ where $E$ 
is a Borel set and $\nabla f (E)$
is the image of $E$ under the gradient map $\nabla f(x)$. 
By Proposition \ref{pr.grad} we have $\mathrm{ri \;} N_\omega \subset \im \nabla R_\omega \subset N_\omega$, 
whence $M  R_\omega (\mathbb{R}^m)=\mathrm{Vol}(N_\omega).$
\end{proof}

\begin{remark}\label{rm.MAmes}
The Monge-Amp\`{e}re measure $M  R_\omega$ can be expressed in terms of supercurrents.
Since $R_\omega(x)$ is convex, one can show that there is a well-defined supercurrent $\frac{1}{m!}(d' d'' R_\omega)^m\in \mathcal{D}^{m,m}(\mathbb{R}^m)$ 
\cite[Section 2]{Lag}.
Let $E$ be a Borel set in $\mathbb{R}^m$ and let $1_E$ be an indicator function of $E$. We can consider $1_E$ as a $(0,0)$-superform on $\mathbb{R}^m.$ 
Then  the Monge-Amp\`{e}re measure $M R_\omega$  of the set $E$ equals 
$$M R_\omega(E) = \frac{1}{m!}(d' d'' R_\omega)^m[1_E].$$
Indeed, if $f(x)$ is a smooth function on $\mathbb{R}^m,$  
then $$(d' d'' f)^m= m! (-1)^{\frac{m(m-1)}{2}} \det(\frac{\partial^2 }{\partial x_j \partial x_k} f) d x_1\wedge\dots\wedge d x_m \otimes d x_1\wedge\dots\wedge d x_m.$$
Therefore  $$\int_{\mathbb{R}^m} 1_E \frac{1}{m!}(d' d'' f)^m =\int_{E} \det(\frac{\partial^2 }{\partial x_j \partial x_k} f) d x_1 \wedge \dots \wedge d x_m=M f(E).$$
One can approximate the function $R_\omega$ by convex smooth functions $f_\varepsilon,$ $f_\varepsilon \rightarrow R_\omega$ as $\varepsilon\rightarrow 0.$
Then $(d' d'' f_\varepsilon)^m$ converges to $(d' d'' R_\omega)^m.$
Hence  $\frac{1}{m!}(d' d'' R_\omega)^m[1_E]$ is equal to the Monge-Amp\`{e}re measure $M R_\omega(E)$ of the set $E$.
\end{remark}

\section{Coordinate-free approach to amoebas.}
\subsection{Basic definitions.} 
In this section we are going to develop a coordinate-free approach to generalized amoebas. It does not provide any additional information about amoebas, 
but nevertheless it looks curious for us.

Let $X$ be a smooth compact complex $n-$dimensional manifold and let $V=\bigcup_j D_j$ be a simple normal crossing divisor on $X.$
We denote the real vector space of imaginary normalized holomorphic differentials on $X\setminus V$ by $H\subset H^0(X\setminus V, \Omega^1),$
and we denote the singular set of a differential from $\omega\in H$ by $\mathrm{Sing} \omega.$

Let $L$ be a real vector space of dimension $m.$ We denote its dual by $L^*,$ and $\langle\cdot,\cdot\rangle$ is the natural paring between $L$ and $L^*.$
Take any linear map $\varphi: L \rightarrow H.$
Consider
$$\mathrm{Sing} \varphi= \bigcup_{\omega \in \im \varphi} \mathrm{Sing} \omega,$$
it is a simple normal crossing divisor on $X$ and $\mathrm{Sing} \varphi\subset V.$ 
Let us fix a point $p_0$ in $X\setminus \mathrm{Sing} \varphi$ and define the following map
$$\mathrm{Log}_{\varphi,p_0}: X\setminus \mathrm{Sing} \varphi \rightarrow  L^*,$$
$$\langle\mathrm{Log}_{\varphi,p_0}(p),l\rangle = \RRe \int^p_{p_0} \varphi(l),$$
where $\langle\mathrm{Log}_{\varphi,p_0}(p),l\rangle$ is the paring between the linear functional $\mathrm{Log}_{\varphi,p_0}(p)\in L^*$ 
and an element $l\in L.$

\begin{definition}The \emph{generalized amoebas} $\mathcal{A}_\varphi$ is the image in $L^*$ of the map $\mathrm{Log}_{\varphi,p_0}.$
\end{definition}
Let $e_1,\dots,e_m$ be a basis of $L.$ Then we can define the isomorphism 
$$\Phi: L^* \rightarrow \mathbb{R}^m,$$
$$\Phi(\xi)= (\langle\xi,e_1\rangle,\dots.\langle\xi,e_m\rangle).$$
Consider $\omega=(\omega_1,\dots,\omega_m),$ where $\omega_j= \varphi(e_j).$
Then  $$\Phi \circ \mathrm{Log}_{\varphi,p_0} = \mathrm{Log}_{\omega,p_0}.$$
In the other words, the definition of the generalized amoeba that was used in the previous sections corresponds to a particular choice of basis in $L.$ 
That is why we consider the approach of this section to be coordinate-free. 

The choice of the point $p_0$ does not play much role. For  $p_0,p'_0\in X\setminus \mathrm{Sing} \varphi$ 
the corresponding maps $\Log_{\omega,p_0},$ $\Log_{\omega,p'_0}$ differ from each other by the shift by a constant vector.
We consider the maps $\Log_{\varphi,p_0}$ for the different choices 
of $p_0$ to be equivalent, and we will usually write just $\Log_{\varphi}$ instead of $\Log_{\varphi,p_0}.$

Let us define the residue along the divisor $D_j$ as follows
$$\mathrm{Res}_{D_j}: L \rightarrow \mathbb{R},$$
$$\mathrm{Res}_{D_j} l = \mathrm{Res}_{D_j} \varphi (l).$$
Thus there is an element $\mathrm{Res}_{D_j}(\cdot) \in L^*$ such that $\langle \mathrm{Res}_{D_j}(\cdot),l \rangle = \mathrm{Res}_{D_j} l.$ Suppose that $\mathrm{Sing} \varphi=\bigcup^s_{j=1} D_j.$ 
If $J$ is a subset of $ \{1,\dots,s\},$ then we can define the cone in $L^*$
$$\Sigma_J =\{\lambda =- \sum_{j\in J} \lambda_{j} \mathrm{Res}_{D_j}(\cdot) \in L^* , \lambda_j\geq 0\}.$$
\begin{definition} The \emph{asymptotic fan} $\Sigma_\varphi$ of the amoeba $\mathcal{A}_\varphi $ is the set of all cones $\Sigma_J$
such that $\bigcap_{j\in J} D_j\neq \emptyset.$
\end{definition}
The linear map $\varphi:L \rightarrow H$ extends naturally on exterior powers of $L,$ we denote by the same symbol
the induced map $$\varphi: \Lambda^n L \rightarrow \Lambda^n H.$$ 
\begin{definition} We say that an amoeba $\mathcal{A}_\varphi$ satisfies the \emph{nondegeneracy condition}
if the map $$\varphi: \Lambda^n L \rightarrow \Lambda^n H \subset H^0(X\setminus V, \Omega^n_{X}(\log V))$$ is not identically zero. 
\end{definition}

\subsection{The supercurrent $S^{m-n,m-n}_\varphi$.} 
We denote by $\mathcal{E}_c^{p,q}(L^*)$ the space of $C^\infty$-smooth differential superforms of degree $(p,q)$ with compact support on $L^*$ (for the definition see Section \ref{s.supercurr}).
Obviously, $$\mathcal{E}_c^{p,q}(L^*) = C_c^\infty(L^*) \otimes \Lambda^p L\otimes \Lambda^q L,$$
where $C_c^\infty(L^*)$ is  $C^\infty$-smooth functions with compact support on $L^*$ and $\Lambda^p L$ is $p$-th exterior power of $L.$
We denote by $\mathcal{D}^{m-p,m-q}(L^*)$ the space of tropical supercurrents of bidegree $(m-p,m-q)$ on $L^*$ . 

Let us define the supercurrent $S^{m-n,m-n}_\varphi\in \mathcal{D}^{m-n,m-n}(L^*).$ 
Since any element of $\mathcal{E}_c^{p,q}(L^*)\cong C^\infty_c(L^*) \otimes \Lambda^p L \otimes \Lambda^q L$ 
can be represented as a linear combination of elements of the following form
$$f(l^*) l\otimes l',$$ where $f(l^*)\in C_c^\infty(L^*),$ $l\in \Lambda^p L,$ $l'\in \Lambda^q L,$ it is enough 
to define values of $S^{m-n,m-n}_\varphi$ on the elements $f(l^*) l\otimes l'.$
Set
$$S^{m-n,m-n}_\varphi[f(l^*) l\otimes l']=\frac{(-1)^n}{2^n}\RRe\int_{X} \frac{1}{(2 \pi i)^n} \Log^*_{\varphi}(f(l^*)) \varphi(l) \wedge \overline{\varphi(l)},$$
here we consider an element $ \varphi(l)\in \Lambda^n H$ as an element of the space of closed holomorphic $n$-forms on $X\setminus V.$

\begin{lemma}
The supercurrent $S^{m-n,m-n}_\varphi$ is $d'$-closed, $d''$-closed, symmetric and positive.
\end{lemma}
The proof of this lemma repeats the proof of Lemma \ref{lm.positive_ronkin}.


\begin{remark}
It seems that the current $S^{m-n,m-n}_\varphi$ may be interesting for a general $m,$ not only when $m=n+1.$
In particular, somehow similar construction was used in \cite{BT} to prove the homological convexity of complements 
of amoebas of algebraic sets of higher codimensions. In complex geometry an complex submanifold $Z$ of dimension $n$ of an $m$-dimensional complex manifold $X$ defines 
the positive integration current $[Z]$ of bidegree $(m-n,m-n).$ We think that the supercurrent  $S^{m-n,m-n}_\varphi$ plays the role of the integration current of the amoeba $\mathcal{A}_\varphi$, 
which is consider as a "tropical manifold" of dimension $n.$
\end{remark}

\subsection{The Ronkin function, the order map and the Newton polytope.}
We are going to rewrite the results of Section \ref{sect.2} in coordinate free terms. Until the end of this section we assume that $$m=n+1.$$ 

Let us fix a nonzero element $\mu$ of $\Lambda^m L,$ we can consider $\mu$ as a volume form on $L^*$ with a constant coefficient. 
Let us define the map $\rho_\mu:L^1_{loc}(L^*) \rightarrow  \mathcal{D}^{0,0}(L^*)$ as
$$\rho_\mu f[\psi]=\int_{L^*} f \psi,$$ where $f\in L^1_{loc}(L^*),$ $\psi \in \mathcal{E}_c^{m,m}(L^*)$ and the integral is understood as the tropical integral 
(the tropical integral depends on the choice of the volume form $\mu$).  Notice that for the volume form $\mu'=c\mu, c\in \mathbb{R}\setminus\{0\},$ we get $\rho_{\mu'}f = \frac{1}{|c|} \rho_{\mu}f.$

\begin{definition} \label{def.ronkin2}The \emph{Ronkin function} $R_\varphi$ of the amoeba $\mathcal{A}_\varphi$ is a function on $L^*$ 
such that $$d' d'' \rho_\mu R_\varphi = S^{1,1}_\varphi.$$
\end{definition}

\begin{proposition}\label{pr.ronkin_eq}
In coordinate terms Definition \ref{def.ronkin2} is exactly the definition of the Ronkin function from the previous 
section (Definition \ref{def.ronkin1}).
\end{proposition}
\begin{proof}
Let $e_1,\dots,e_m$ be a basis of $L.$ Then there is an isomorphism
$$\Phi: L^* \rightarrow \mathbb{R}^m,$$
$$\Phi(l^*)= (\langle l^*,e_1\rangle,\dots.\langle l^*,e_m\rangle).$$

The isomorphism $\Phi$ induces the isomorphisms: $$\Phi_*: \mathcal{D}^{p,q}(L^*) \rightarrow \mathcal{D}^{p,q}(\mathbb{R}^m),\Phi^*: \mathcal{E}^{p,q}(\mathbb{R}^m) \longrightarrow \mathcal{E}^{p,q}(L^*).$$ 
The following relation holds $(\Phi_* \psi)[ \xi]=\psi[ \Phi^* \xi],$ where $\psi\in \mathcal{D}^{m-p,m-q}(L^*)$
 and $\xi\in  \mathcal{E}_c^{p,q}(\mathbb{R}^m).$

Consider $\omega=(\omega_1,\dots,\omega_m)$, where $\omega_j=\varphi(e_j).$ 
Then $ \Phi\circ\mathrm{Log}_{\varphi}=\mathrm{Log}_{\omega}$ and $\Phi\mathcal{A}_\varphi=\mathcal{A}_\omega\subset \mathbb{R}^m.$
A superform $\psi\in\mathcal{E}_c^{n,n}(\mathbb{R}^m)$ can be written as 
$$\psi=\sum^m_{j,k=1} f_{jk}(x) d x[j]\otimes d x[k],$$
where $$ d x[j] = d x_1\wedge \dots \wedge d x_{j-1}\wedge d x_{j+1} \wedge \dots \wedge d x_m.$$
Since $\varphi \Phi^* d x[j] = \Omega_j,$ where 
$$\Omega_j =\omega_1\wedge \dots \wedge \omega_{j-1}\wedge \omega_{j+1} \wedge \dots \wedge \omega_m,$$
we get $$(\Phi_* S^{1,1}_\varphi)[\psi]=S^{1,1}_\varphi[\Phi^* \psi]=\sum^m_{j,k=1} 
\frac {(-1)^{n}}{2^n} \RRe \int_{X} \frac{1}{(2 \pi i)^n} 
\Log^*_\omega (f_{jk}) \Omega_{j}\wedge\overline{\Omega}_{k}.$$
Therefore by Definition \ref{def.ronkin1} the Ronkin function $R_\omega$ of the amoeba $\mathcal{A}_\omega$ is exactly equal to
$$d' d'' R_\omega = (\Phi_* S^{1,1}_\varphi),$$
thus we get $$d' d'' \rho_\mu \Phi^* R_\omega = S^{1,1}_\varphi,$$
where $\mu=\Phi^* d x.$

\end{proof}

Therefore we can rewrite Theorem \ref{th.ronkin} as follows
\begin{theorem}
  The following statements hold: 
  \begin{enumerate}
  \item The Ronkin function $R_{\varphi}$ exists and it is unique up to addition of an affine function, it is a continuous convex function on $L^*.$
  \item Suppose that the amoeba $\mathcal{A}_\varphi$ satisfies the nondegeneracy condition. 
  Then for any connected open set $C\subset L^*$ the restriction of $R_\varphi$ to $C$ is affine if and only if $C$ 
  does not intersect the amoeba $\mathcal{A}_\varphi$. 
  \end{enumerate}
\end{theorem}

We can consider $\mathcal{D}^{0,0}(L^*)\otimes L$ as a $(0,0)$-currents with coefficients in $L,$ hence for any $T\in \mathcal{D}^{0,0}(L^*)\otimes L$ and any $ \psi\in \mathcal{E}_c^{m,m}(L^*),$
$T[\psi]$ is an element of $L.$ 

Let $e_1,\dots,e_m$ be a basis of $L$.
We consider $e_j\in L$ as a differential $1$-form on $L^*$ with constant coefficients. 
Let us denote $$e=e_1\wedge\dots \wedge e_m$$ and $$e_{[j]}=e_1\wedge\dots\wedge e_{j-1} \wedge e_{j+1}\wedge \dots \wedge e_m.$$
Any element $\psi\in \mathcal{E}_c^{m,m}(L^*)$ can be written as 
$\psi= f(l^*) e \otimes e$, where  $f(l^*)\in C_c^\infty(L^*).$
Let us define the isomorphism $$\pi_1:\mathcal{D}^{1,0}(L^*) \rightarrow \mathcal{D}^{0,0}(L^*)\otimes L$$
as
$$\pi_1(\varphi)[f(l^*) e\otimes e ]=\sum^m_{j=1} (-1)^{j+1} \varphi[f(l^*) e_{[j]}\otimes e ] \otimes e_j.$$
One can check that this definition does not depend on the choice of the basis.

Let $C$ be a connected component of $L^*\setminus \mathcal{A}_\varphi$ and let $\psi_C$ be an $(m,m)$-superfrom,
$\psi_C\in \mathcal{E}_c^{m,m}(L^*),$ such that $\supp  \psi_C \subset C$ and $\int_{L^*} \psi_C =1$ (we need a volume form $\mu\in \Lambda^m L$ to define this tropical integral).
We denote the set of connected components of $L^*\setminus \mathcal{A}_\varphi$ by $\Upsilon.$ 
Then we define the order map $\nu^\mu_\varphi:$
\begin{definition}
The \emph{order map} $$\nu^\mu_\varphi: \Upsilon \rightarrow L$$
is defined as
$$\nu^\mu_\varphi(C) = (\pi_1 d'\rho_\mu R_\varphi)[\psi_C].$$
\end{definition}
\begin{lemma}\label{lm.welldef}
 The order map $\nu^\mu_\varphi$ is well-defined.
\end{lemma}
\begin{proof}
We should check that $\nu^\mu_\varphi(C)$ does not depend on the choice of the form $\psi_C.$
Suppose that $\psi^1_C,\psi^2_C\in \mathcal{E}_c^{m,m}(L^*)$ are two $(m,m)$-superforms such 
that $\supp  \psi^j_C \subset C$ and $\int_{L^*} \psi^j_C =1$ for $j=1,2.$ Then $\psi^1 - \psi^2$ is $d''-$exact and there is a form $\psi\in \mathcal{E}_c^{0,m-1}(L^*)$ such that $\supp  \psi \subset C$ and
$$e\wedge d'' \psi = \psi^1 - \psi^2.$$ 
\begin{multline*}
(\pi_1 d'\rho_\mu R_\varphi)[\psi^1_C]- (\pi_2 d'\rho_\mu R_\varphi)[\psi^1_C]= (\pi_1 d'\rho_\mu R_\varphi)[e\wedge d'' \psi]= \sum^m_{j=1} (-1)^{j+1} d'\rho_\mu R_\varphi[e_{[j]}\wedge d''  \psi ] \otimes e_j=\\
=(-1)^{m} \sum^m_{j=1} (-1)^{j} d'' d'  \rho_\mu R_\varphi[e_{[j]}\wedge  \psi ] \otimes e_j =0 
\end{multline*}
The last equality holds because $\supp d' d'' \rho_\mu R_\varphi \subset \mathcal{A}_\varphi$ does not intersect $\supp  \psi.$
\end{proof}

\begin{definition} The \emph{Newton polytope}  $N^\mu_\varphi\subset L$ of the amoeba $\mathcal{A}_\varphi $  
is the convex hull of the image of $\nu^\mu_\varphi.$ 
\end{definition}
We define the \emph{normal cone} to $N^\mu_\varphi$ at a point $p\in L$ to be equal to 
$$\mathrm{norm}_{N^\mu_\varphi}(p)=\{l^*\in L^*: \forall l \in N^\mu_\varphi: \langle l-p, l^*\rangle \leq 0\}.$$

We can rewrite Theorem \ref{th.rec} in a new terms as follows 
\begin{theorem}\label{th.recc2}
The order map $\nu^\mu_\varphi$ is injective.
Let $C$ be a connected component of $L^*\setminus \mathcal{A}_\varphi.$ Then the recession cone of $C$ is equal to the normal cone to the Newton 
polytope $N^\mu_\varphi$ at the point $\nu^\mu_\varphi(C).$ 
\end{theorem}
\begin{proof}
It follows directly from the expression of $(\pi_1 d'\rho_\mu R_\varphi)[\psi_C]$ in terms of coordinates.
\end{proof}

Remark \ref{rm.MAmes} can be rewritten in the coordinate-free case. 
Since $R_\varphi$ is convex, one can show that there is a well-defined supercurrent $\frac{1}{m!}(d' d'' \rho_\mu R_\varphi)^m\in \mathcal{D}^{m,m}(L^*).$  
Let $M R_\varphi$ be a Monge-Amp\`{e}re measure of $R_\varphi,$  then we have 
$$M R_\varphi(E) = \langle  \frac{1}{m!}(d' d'' \rho_\mu R_\varphi)^m , 1_E \rangle.$$
Then  Proposition \ref{pr.totalmass} can be written in a coordinate-free form as follows 
$$M R_\varphi(L^*)=\mathrm{Vol} N^\mu_\varphi.$$

\end{document}